\documentclass[a4paper,reqno,11pt]{amsart}
\usepackage{epsfig}
\usepackage[pdftex]{hyperref}
\usepackage{amsfonts}
\usepackage{tikz}
\tikzstyle{b_vertex}=[circle,fill=black!100,text=white,inner sep=0.8mm,draw]
\tikzstyle{w_vertex}=[circle,fill=white!100,text=white,inner sep=0.8mm,draw]
\tikzstyle{point}=[circle,fill=black,inner sep=0.1mm]
\tikzstyle{path_edge}=[thick]
\usetikzlibrary{decorations.pathreplacing}

\usepackage{mathtools}
\usepackage{amsmath,amscd}
\usepackage{fullpage}
\usepackage{amssymb}
\usepackage{amsthm}
\usepackage{mathrsfs}
\usepackage{graphicx, wrapfig, caption, subcaption}
\usepackage{comment}

\usepackage{enumitem}
\usepackage{xcolor}
\usepackage{svg}

\usepackage{steinmetz}
\usepackage{cleveref}
\usepackage{placeins}
\theoremstyle{plain}
\newtheorem{theorem}{Theorem}[section]

\newtheorem{definition}[theorem]{Definition}
\newtheorem{lemma}[theorem]{Lemma}
\newtheorem{corollary}[theorem]{Corollary}
\newtheorem{proposition}[theorem]{Proposition}
\newtheorem{Conjecture}[theorem]{Conjecture}

\newtheorem{example}[theorem]{Example}
\newtheorem{remark}[theorem]{Remark}
\newtheorem{question}[theorem]{Question}
\newtheorem{prob}[theorem]{Problem}

\newtheorem{innercustomclaim}{Claim}
\newtheorem{innercustomthm}{Theorem}
\newenvironment{customclaim}[1]
  {\renewcommand\theinnercustomclaim{#1}\innercustomclaim}
  {\endinnercustomclaim}

\newtheorem*{claim*}{Claim}
\newtheorem*{observation*}{Observation}

\newenvironment{customthm}[1]
  {\renewcommand\theinnercustomthm{#1}\innercustomthm}
  {\endinnercustomthm}

\newcommand{\ack}[1][\ackname]{\subsection*{#1}}

\theoremstyle{definition}

\newcommand{\diam}{\mathrm{diam}}
\newcommand{\conv}{\mathrm{conv}}

\newcommand{\Lk}{\mathrm{Lk}}

\newlist{thmenum}{enumerate}{1} 
\setlist[thmenum]{label=\textup{(\roman*)},
                  ref=\thetheorem.\textup{(\roman*)}}
\crefname{thmenumi}{theorem item}{theorem items}

\newcommand{\R}{\mathbb{R}}
\newcommand{\RR}{\mathcal{R}}
\newcommand{\Sc}{\mathcal{S}}

\def\s{\mathbb{S}}
\def\D{\mathcal{D}}

\def\O{\mathcal{O}}

\newcommand{\sub}[1]{{\left\langle{#1}\right\rangle}}
\numberwithin{equation}{section}
\begin{document}
	\title{On the Classification of Planar-Rips complexes and their corresponding unit disk graphs}
	\vspace{2cm}
	
	 \author{Vinay Sipani}
\address{Department of Mathematics\\ Indian Institute of Technology\\ Chennai - 600036 \\ India.}
\email{sipanivinay@gmail.com; vinaysipani@alumni.iitm.ac.in}
	
\author{Ramesh Kasilingam}
	
	\address{Department of Mathematics\\ Indian Institute of Technology\\ Chennai - 600036 \\ India.}
\email{rameshkasilingam.iitb@gmail.com  ; rameshk@iitm.ac.in}
	
	\date{}
	
	\begin{abstract}
Given a metric space $(X,d)$, the Vietoris-Rips complex of $X$ at a scale of $r >0$ is a simplicial complex whose simplices are all those finite subsets of $X$ with diameter less than $r$. In this paper, we classify, up to simplicial isomorphism, all $n$-dimensional pseudomanifolds and weak-pseudomanifolds that can be realized as a Vietoris-Rips complex of planar point sets. We further classify two-dimensional, pure, and closed planar-Rips complexes up to homotopy. Additionally, we explore the hereditary properties and introduce the notion of obstructions in planar-Rips complexes. We also consolidate our findings to describe a class of unit disk graphs, having all maximal cliques of same cardinality. Several structural and geometric properties of planar-Rips complexes have also been derived.
	\end{abstract}
\maketitle
\section{Introduction}
Given a metric space $(X,d)$ and a scale parameter $\delta>0$, the Vietoris-Rips complex $\RR(X;\delta)$ is a simplicial complex whose simplices are all those finite subsets of $X$ with diameter less than $\delta$, i.e., $\RR(X;\delta)=\{A \subseteq X \colon |A|<\infty \text{ and } d(x,y)<\delta \  \forall x,y \in A\}$.The Vietoris-Rips complex is a fundamental construct in computational topology, serving as a bridge between discrete data and continuous shape analysis. First introduced by Vietoris \cite{Vietoris} in the context of his work on (co)homology theory and covering spaces, and later independently by Rips in the context of geometric group theory \cite{Rips}, both used it essentially to understand complex topological spaces through a combinatorial framework. In this paper, we will only be working within the Euclidean setting, and all our points sets will be finite subsets of $\R^2$. Henceforth, we shall refer to them as \textbf{planar-Rips complexes} \cite{planar1}.

\begin{definition}
     A finite simplicial complex $K$ with the vertex set $V$ is said to admit a planar-Rips structure if there exists an injective map $f : V \to \R^2$ such that $K$ is isomorphic to the Vietoris-Rips complex $\RR(f(V);\delta)$ for some $\delta>0$ under the usual Euclidean metric.
\end{definition}
This paper is specifically streamlined towards determining the simplicial complexes that (do not) admit a planar-Rips structure. Therefore, with the scale-invariance under consideration, we fix $\delta=1$ throughout and use the notation $\RR(X)$ instead of $\RR(X;1)$.
Our study is primarily driven by the works of \cite{hausmann,latschev,planar1,planar2,circle} where the respective authors have widely addressed the characterization problem of planar-Rips complexes at various scales. 

In his paper \cite{hausmann}, Hausmann demonstrated that up to homeomorphism, every $n$-sphere admits a planar-Rips structure for some subset of circle $\s^1$. This further implies that for a given wedge of spheres, there exists a subset $X \subset \R^2$ whose planar-Rips complex is homotopy equivalent to the given wedge.  In \cite{planar1}, Chambers et al. prove that the planar-Rips complexes have free fundamental group. As a consequence, one can easily verify that every finite two-dimensional planar-Rips complex is homotopy equivalent to wedge of circles and $2$-spheres.

In \cite{circle}, Adamaszek and Adams obtain a wide range of homotopy types of planar-Rips complexes for subsets of circle $\s^1$. In particular, Adamaszek et al. in \cite{nerve} show that the planar-Rips complex of any finite subset of $\s^1$ is homotopy equivalent to either a point, an odd sphere, or a wedge sum of spheres of the same even dimension. Further in \cite{planar2}, Adamaszek et al. show that the cross-polytopal spheres are the only normal pseudomanifolds that admit planar-Rips structure. With this, they postulate the following conjecture, a problem which still remains open: 
\begin{Conjecture} \cite[Problem 7.3]{planar2}
    For any $X \subset \R^2$, the planar-Rips complex $\RR(X)$ is either contractible or homotopy equivalent to wedge of spheres.
\end{Conjecture}

In this spirit, we extend the planar-Rips classification to a larger class of pure simplicial complexes, thereby providing an affirmative evidence supporting the conjecture. Our main results are as follows:

\begin{customthm}{A} (See Theorem \ref{octa} $\&$ \ref{pseudorips})\label{A}
  The n-dimensional cross-polytopal spheres ($n\geq 2$) are the only pseudomanifolds that admit planar-Rips structure.
\end{customthm} 

This will be an inductive proof. We use this to further give a planar-Rips classification of $n$-dimensional weak-pseudomanifolds, each of whose strongly connected component is a pseudomanifold. (See Section \ref{sec3} for definition of pseudomanifold and weak-pseudomanifold)

\begin{customthm}{B} (See Theorem \ref{weakiswedge} $\&$ \ref{weakhomotopy}) \label{B}
    An n-dimensional weak pseudomanifold ($n \geq 2)$ admitting a planar-Rips structure is isomorphic to an iterated chain of $n$-dimensional cross-polytopal spheres. Further, they are homotopy equivalent to wedge sum of circles and $n$-spheres.
\end{customthm}

The tools from \cite{planar1,planar2} and the proximity relations derived towards proving the above two theorems are further applicable towards the characterization of two-dimensional pure and closed simplicial complexes that admit planar-Rips structure.

\begin{customthm}{C} (See Theorem \ref{minimal}  $\&$ Proposition \ref{b0}) \label{C}
    Suppose $K$ is a two-dimensional pure and closed simplicial complex admitting a planar-Rips structure, then
    \begin{enumerate}[label=(\alph*)]
\item $K$  contains a copy of octahedron as an induced subcomplex.
\item The number of copies of octahedron as induced subcomplex in $K$ is same as its second Betti number $b_2(K)$.
\end{enumerate}
\end{customthm}

As a simple consequence of the above theorem, if $K$ is a minimal $2$-cycle, then it isomorphic to octahedron.  Furthermore, $K$ is homotopy equivalent to wedge of circles and $2$-spheres.

The planar-Rips complexes are also useful from an application standpoint because of their ability to faithfully represent proximity-based relationships in the Euclidean plane. Recall that a (unit) disk graph is the intersection graph of (unit) equal-radius disks in the Euclidean plane. Notably, a natural isomorphism exists between the category of planar-Rips complexes and the category of disk graphs, with disk graphs being precisely the 1-skeleton of planar-Rips complexes. For more on (unit) disk graphs, we refer the reader to \cite{udg1,udg3,udg5,udg6}

In light of the isomorphism, we use our findings to characterize a class of disk graphs with a uniform clique structure. In particular, the graph analogs to Theorems \ref{A},\ref{B},\ref{C} have been substantiated in Section \ref{applications}, which significantly streamline the graph recognition process for disk graphs.

Another intriguing aspect of planar-Rips complexes is that they satisfy hereditary property. This means that every induced subcomplex is also a planar-Rips complex. It is well established that any hereditary class of simplicial complexes can be fully characterized in terms of minimal forbidden induced subcomplexes called \textbf{obstructions}. These obstructions represent critical subcomplex patterns that cannot appear as induced subcomplexes within any complex in the class. For example, the hereditary class of flag complexes can also be defined as the simplicial complexes that do not contain the boundary complex of any $n$-simplex ($n>1$) as an induced subcomplex.  The study of obstructions and
hereditary properties are found in \cite{obs1, obs2, obs3, obs4}.  

In essence, any simplicial complex that does not admit a planar-Rips structure further contains a `minimal' forbidden induced subcomplex. Following Theorems \ref{A},\ref{B},\ref{C}, we have indeed obtained a wide spectrum of simplicial complexes that do not admit a planar-Rips structure. This opens up the possibility of obtaining a large collection of obstructions of planar-Rips complexes by extracting each one as an induced subcomplex from the forbidden planar-Rips structures. The feasibility of establishing such a framework has been demonstrated by a few examples in Section \ref{applications}. For example, Theorem \ref{obsinrp2} substantiates an obstruction of planar-Rips complex, derived from a minimal flag triangulation of the projective space $\R P^2$.

\ack

We thank Henry Adams and Jan Segert for making the \textit{Mathematica Demonstration on Vietoris-Rips and $\check{C}ech$ complexes} publicly available \cite{nb}, which aided in refining our approaches. We also thank Priyavrat Deshpande and Anurag Singh for introducing the notion of obstructions to shellability during the NCMW workshop \textit{``Cohen-Macaulay Simplicial Complexes in Graph Theory"}, which inspired the extension of this concept for planar-Rips complexes.

\subsection{Organization of the paper}
In Section \ref{prelim}, a background overview and key topological concepts are introduced relevant to this paper. Section \ref{sec3} delves into developing the necessary geometric aspects to prove Theorem \ref{A} and Theorem \ref{B}. Section \ref{sec4} focuses on establishing the structural characterization of two-dimensional closed planar-Rips complexes as part of proving Theorem \ref{C}. Lastly, in Section \ref{applications}, we consolidate our findings on planar-Rips complexes in terms of unit disk graphs.  In addition, the hereditary property of planar-Rips complexes has been explored and the notion of obstructions of planar-Rips complexes has been conceptualized with few illustrative examples.

\section{Preliminaries} \label{prelim}

Throughout this paper, we consider the point set $X$ to be a finite subset of $\R^2$ and the planar-Rips complex $\RR(X)$ to be connected. Further, the geometric realization of $\RR(X)$ is considered while referring to the associated topological properties of $\RR(X)$. We say two points/vertices are adjacent in $\RR(X)$ if the distance between them is less than $1$ i.e., there is an edge between them in $\RR(X)$.

We shall also use the following notation: Given a simplicial complex $K$ on a vertex set $V$, the \textit{link of a simplex} $\sigma$ in $K$ is defined to be $\left\{\tau \in K \mid \tau \cap \sigma=\emptyset, \tau \cup \sigma \in K \right\}$, and is denoted by $\Lk(\sigma)$. For a subset $U$ of $V$, denote the induced subcomplex of $K$ on $U$ by $\sub{U}$.

We now define the concept of ``shadow" of a planar-Rips complex $\RR(X)$ inside $\R^2$. This notion was originally introduced by Chambers et al.\cite{planar1} to simplify the computability. Likewise, we employ the concept of shadow to simplify the representation of proximity relations and certain topological properties of planar-Rips complexes.

\begin{definition}[Shadow, projection map \cite{planar2}] Given a finite set $X \subset \mathbb{R}^2$, the shadow $\Sc(X)$ is the image of the planar-Rips complex $\RR(X)$ under the projection map $p \colon \RR(X) \to \R^2$ which maps each vertex in $X$ to the corresponding point inside $\R^2$, and extends linearly to the simplices of $\RR(X)$. More precisely,
     \[\Sc(X)\coloneq \bigcup_{\substack{Y\subseteq X, \ |Y|\leq 3\\
                  \diam(Y)<1}} \conv \ Y\]
\end{definition}

Note that by Caratheodory's theorem, the convex hull of any finite set $Y$ in $\R^2$ is equal to the union of convex hulls of at most 3-element subsets of $Y$. Henceforth, the condition $|Y| < \infty$ is redundant to $|Y| \leq 3$ in the above defintion. 
\begin{example}
Figure \ref{shadow} illustrates a point set $X$ in $\R^2$, forming a four-dimensional planar-Rips complex $\RR(X)$. Along with it is a representation of projection map onto its shadow $\Sc(X)$.
\begin{figure}
\centering
  \includegraphics[scale=0.25]{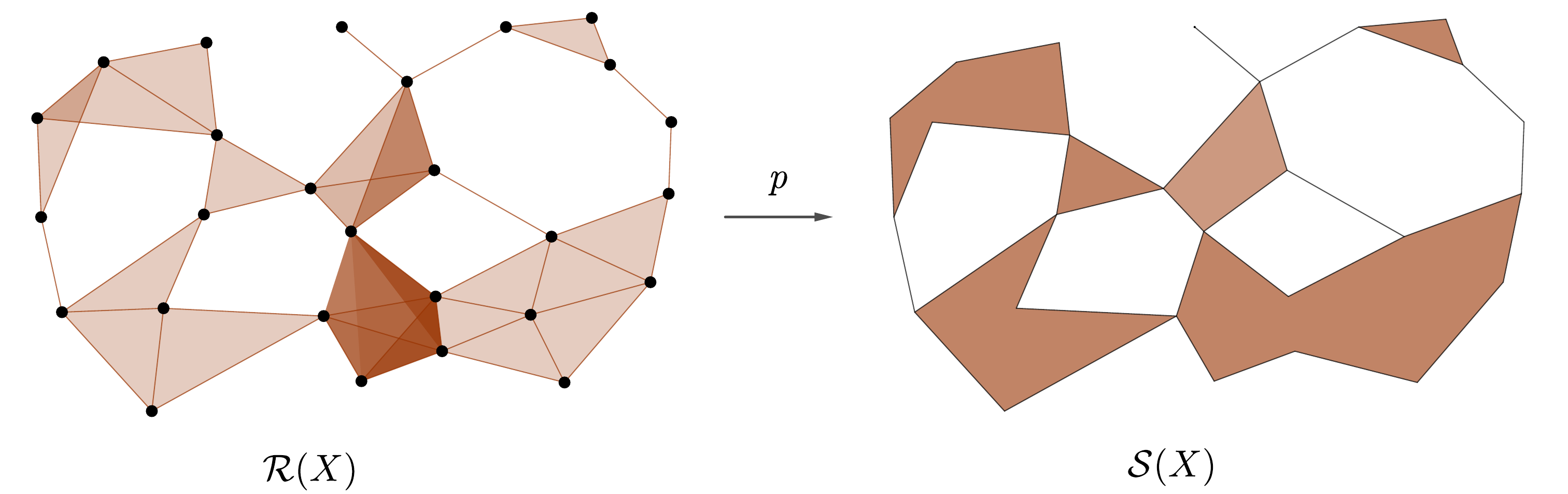}
  \caption{An illustration of projection of a planar-Rips complex $\RR(X)$ onto its shadow $\Sc(X)$ inside Euclidean plane.}
  \label{shadow}
\end{figure}
\end{example}

 The following result by Chambers et al. in \cite{planar1} depicts the topological faithfulness of the projection map $p$, representing the fundamental group of $\RR(X)$ in terms of its shadow $\Sc(X)$. 
\begin{theorem}\cite{planar1} \label{free}
For any finite subset $X \subset \R^2$, the projection map $p\colon\RR(X) \to \Sc(X)$ induces a $\pi_1$-isomorphism, hence implying that the fundamental group $\pi_1 (\RR(X))$ is free.
\end{theorem}

It has been readily apparent in our study that the applicability of proximity relations between the points is quite inevitable for the study of planar-Rips complexes. Hence, there is a need to deduce many more such properties in pertinence to the geometry of $\R^2$. The following are some key properties of planar-Rips complexes that will be referenced throughout this paper.

\begin{lemma} \cite[Proposition. 2.1]{planar1}\label{cone}  Suppose $X\subset \R^2$ and let $A,B,C,D$ be four distinct points in $X$ such that the point $A$ is adjacent to $B$, and $C$ is adjacent to $D$ but not $A$ in $\RR(X)$. Further, the segments $\overline{AB}$ and $\overline{CD}$ intersect in the shadow $\Sc(X)$. Then $B$ is adjacent to $D$ in $\RR(X)$.
\end{lemma}

\begin{lemma} \label{k16}
The graph of $1$-skeleton of a planar-Rips complex $\RR(X)$ cannot contain an induced subgraph isomorphic to the bipartite graph $K_{1,6}$
\end{lemma}

The above result is a direct consequence of \cite[Lemma 3.2]{udg1}, and has been elaborately explained in Section \ref{applications}.

\begin{lemma} \label{triangle}
Given $X \subset \R^2$, let $A$ be a subset of $X$ and $\sigma$ be a simplex such that $A \subseteq \Lk(\sigma)$ inside $\RR(X)$. Suppose a simplex $\gamma$ is in the convex hull of $A$, then $\gamma \in \Lk(\sigma)$ and every vertex of $\gamma$ is adjacent to at least one vertex of $A$ in $\RR(X)$.
\end{lemma}
\begin{proof}
    Let $u$ and $v$ be any vertex in $\sigma$ and $\gamma$ respectively. Firstly, by Caratheodory's theorem, there exists vertices $a_1,a_2\in A$ (not necessarily distinct) such that $v \in conv(u,a_1,a_2)$. Now, since $d(u,a_i)\leq 1$ for all $i=1,2$, we have $d(u,w)\leq 1$ for any $w \in  conv(u,a_1,a_2)$. Since this is true for any $u \in \sigma$, we have $\gamma \in \Lk(\sigma)$. 

    Now consider the unit circles inside $\R^2$ centered at points $a_1,a_2$. Since the union of these circles covers $conv(u,a_1,a_2)$, the point $v$ is adjacent to at least one of $a_1, a_2$.
\end{proof}
\begin{remark}
    The above result generally holds even for the case when $X \subset \mathbb{R}^n$ under the usual Euclidean metric, and a similar proof is applicable.
\end{remark}

\begin{proposition} \label{hull}

For a Vietoris-Rips complex on planar point sets, the convex hulls of any two connected components in the link of a vertex cannot intersect in the shadow $\Sc(X)$.   
\end{proposition}
\begin{proof}
Suppose $\mathcal{C}_1$ and $\mathcal{C}_2$ are any two components of link of a vertex inside $\RR(X)$. By Lemma \ref{triangle}, one of the convex hulls of $V(\mathcal{C}_1)$ and $V(\mathcal{C}_2)$ cannot contain in the other. Further, from Lemma \ref{cone}, the convex hulls cannot intersect on their boundaries.
\end{proof}

\section{Classification of Rips complexes with weak-pseudomanifold structure} \label{sec3}

In this section, we present a complete classification of well-behaved spaces known as pseudomanifolds and weak-pseudomanifolds admitting a planar-Rips structure. To accomplish this, we first introduce the necessary terminologies and tools before proceeding with the proof of Theorem \ref{A} by induction. Subsequently, we use this to further prove Theorem \ref{B}.

The dimension of a simplex will be one less than its cardinality. The \textit{dimension} of a simplicial complex is defined to be equal to the supremum of all $n$'s such that $K$ has a face of dimension $n$. A simplex in $K$ is called \textit{maximal} or a \textit{facet} if it is not contained in another larger simplex. An $n$-dimensional simplicial complex is said to be \textit{pure} if every maximal simplex is of dimension $n$. We refer the reader to look at \cite{bjorner,shastri} for more on these topics.
\begin{definition}[Strongly connected simplicial complex] A pure simplicial complex is said to be strongly connected if for any two facets $\sigma$ and $\gamma$, there exists a sequence of $n$-simplices, say $\lambda_1,\lambda_2,\dots \lambda_k,$ such that $\lambda_1=\sigma,\ \lambda_k=\gamma$ and $\lambda_i \cap \lambda_{i+1}$ is an $(n-1)$-simplex for all $1\leq i \leq k-1$.
\end{definition}

We now recursively define the classes of pure simplicial complexes viz. weak-pseudomanifolds, pseudomanifolds, and normal pseudomanifolds
\begin{definition} \label{defpseudo}
\begin{enumerate}[label=(\alph*)]
    \item An $n$-dimensional simplicial complex is said to be a \textit{weak pseudomanifold} if it is pure and each $(n-1)$-simplex is contained in exactly two $n$-simplices.
    \item An $n$-dimensional weak-pseudomanifold is called a \textit{pseudomanifold} if it is strongly connected.
    \item An $n$-dimensional pseudomanifold $(n\geq 2)$ is said to be a \textit{normal pseudomanifold} if the link of each simplex of codimension at least two is connected.
\end{enumerate} 
\end{definition}

The following result by Adamaszek et al. in \cite{planar2} gives a complete classification of normal pseudomanifolds that admit a planar-Rips structure. 

\begin{theorem} \cite[Theorem 6.3]{planar2} \label{normal} 
For any $n \geq 2$ and any finite subset $X \subset \R^2$, if the planar-Rips complex $\RR(X)$ is an $n$-dimensional normal pseudomanifold then $\RR(X)$ is isomorphic to boundary complex of the $(n + 1)$ dimensional cross-polytope.
\end{theorem}
Hence, apart from crosspolytopal spheres, no other manifold admits a planar-Rips structure inside $\R^2$. But it is interesting to determine the class of manifolds that admit a planar-Rips structure at least upto homotopy equivalence. For instance, the methodology in proving whether or not the manifold $\s^1 \times \s^2$ admits a planar-Rips structure upto homotopy equivalence will itself make a pavement of myriad tools to study the classification of planar-Rips structures inside $\R^2$.

A simplex of a finite weak-pseudomanifold of codimension at least two will be called \textit{singular} if its link is not connected. The proposition below provides an implication of the precise number of connected components in the link of a singular vertex when $\RR(X)$ is a weak-pseudomanifold.

\begin{proposition} \label{cycles}
    For $X\subset \R^2$, suppose the planar-Rips complex $\RR(X)$ is an $n$-dimensional weak-pseudomanifold where $n\geq 2$, then the link of any simplex of codimension at least two can have at most two strongly connected components.
\end{proposition}
\begin{proof}
On contrary, suppose there are more than two strongly connected components in the link of some $k$-simplex, say $\sigma$, where $0\leq k \leq n-2$ . We have that the link of the simplex  $\sigma$ is again a weak-pseudomanifold of dimension $n-k-1$. Since each strong component is a pseudomanifold and admits a planar-Rips structure, choose two non-adjacent vertices from each of them, hence forming an induced subgraph $K_{1,6}$ . This is not possible owing to Lemma \ref{k16}.
\end{proof}
\begin{corollary}
   For $X\subset \R^2$, if $\RR(X)$ is an $n$-dimensional weak-pseudomanifold ($n\geq 2$), then the link of any $(n-2)$ simplex can have at most 11 vertices. Further, it is either a cycle or a disjoint union of exactly two cycles.
\end{corollary}

The following lemma from \cite{pseudo} provides a general representation of all two-dimensional weak-pseudomanifolds, specifically for two-dimensional pseudomanifolds by restricting to the strongly connected components.

\begin{lemma}[Classification of two–dimensional weak-pseudomanifolds] \label{pseudo}  Any compact two-dimensional weak-pseudomanifold $X$ is homeomorphic to the quotient space of a disjoint union of compact 2–manifolds without boundary such that if $q$ is the quotient map, then $\#\left\{q^{-1}(x)\right\} < \infty$ for all points $x \in X$, and $\#\left\{q^{-1}(x)\right\} = 1$ for all but finitely many $x \in X$.
\end{lemma}
 Thus, a pseudomanifold $X$ is the quotient space of a compact 2-manifold without boundary by taking a self-identification of finite sets of points. With the aim of classifying all planar-Rips complexes having a weak-pseudomanifold structure, we first extend Theorem \ref{normal} for a pseudomanifold. We begin by proving the two-dimensional case.

\begin{theorem}\label{octa}
For any $X \subset \R^2$, the planar-Rips complex on $X$ having a two-dimensional pseudomanifold structure is isomorphic to the boundary complex of octahedron.
\end{theorem}

\begin{proof}

\begin{figure}[h!]
  \begin{subfigure}[b]{0.49\textwidth}
    \centering \includegraphics[scale=0.225]{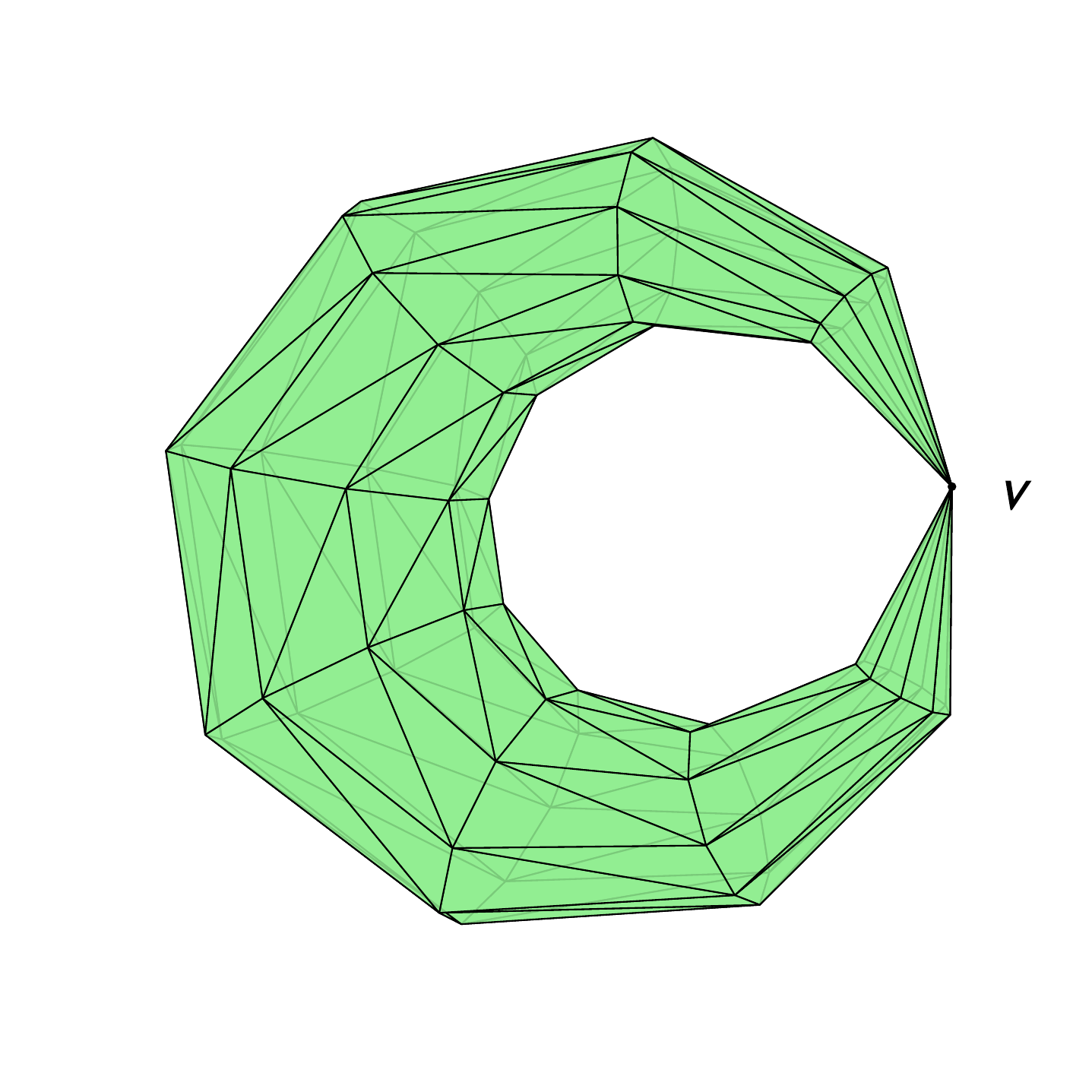}
    \caption{}
  \end{subfigure}
  \begin{subfigure}[b]{0.49\textwidth}
    \centering 
    \includegraphics[scale=0.22]{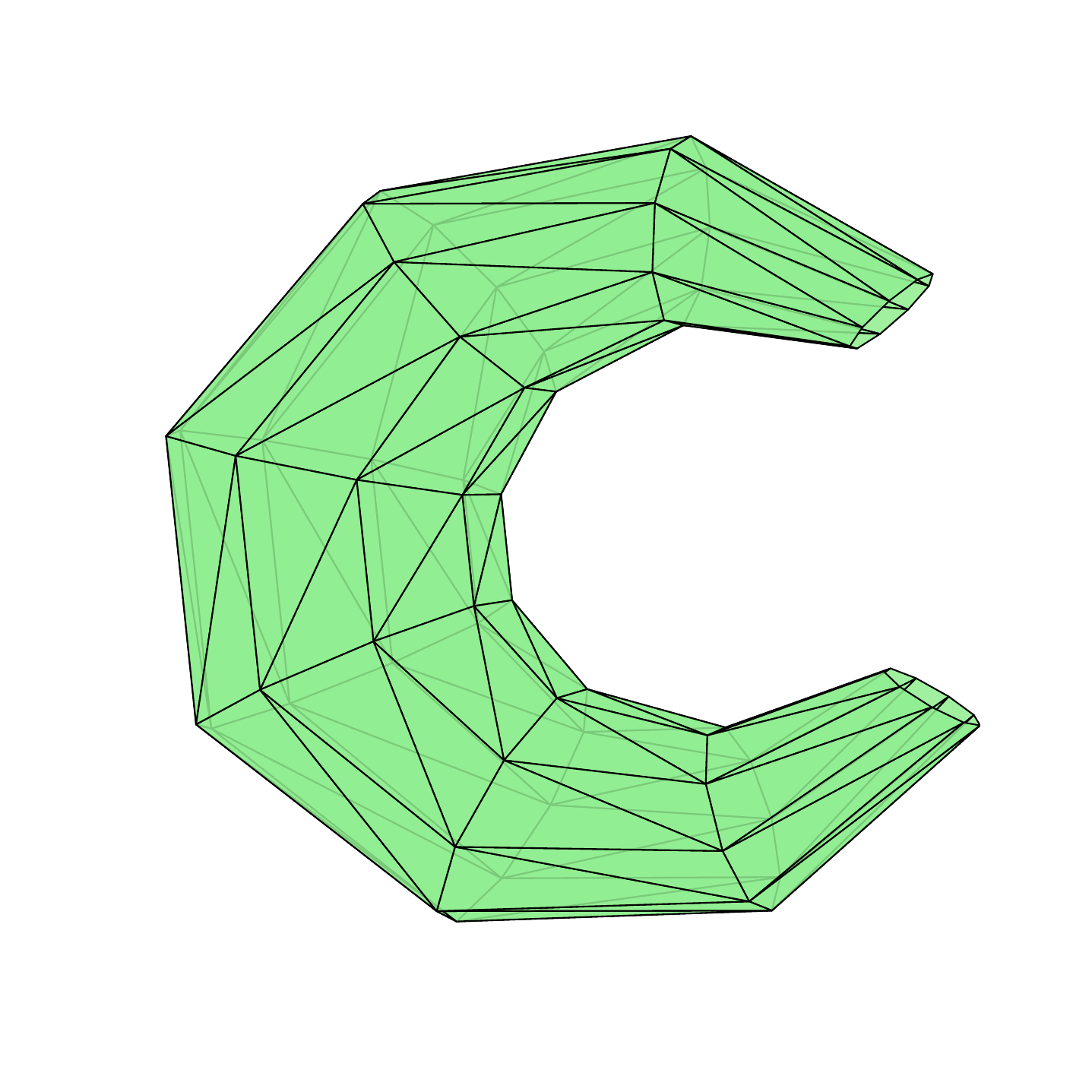}
    \caption{}
  \end{subfigure}
  \caption{ (On the left) A sphere with self-identification of two points($m=1$) is homeomorphic to pinched torus. The link of its singular vertex $\mathrm{v}$ has two connected components. (On the right) We have induced subcomplex of pinched torus after deleting the singular vertex. Both the structures do not admit planar-Rips structure.} \label{pinchedtorus}
\end{figure}
By Theorem \ref{free}, the fundamental group of $\RR(X)$ is free. Hence, by the structure theroem for (weak) pseudomanifolds (Lemma \ref{pseudo}), the pseudomanifold $\RR(X)$ is homeomorphic to the quotient space of sphere $\s^2$ by taking a self-identification of finite sets of points; let $q\colon\s^2 \to \RR(X)$ be the associated quotient map. 

Let $\tilde{X}$ be the set of \textit{singular vertices} (vertices whose links are not connected) in $\RR(X)$.  For any vertex $v \in X$, since $\Lk(v)$ has at most two disjoint cycles (by Proposition \ref{cycles}), the set $q^{-1}(v)$ can have at most two points, and $q^{-1}(u) \cap q^{-1}(v) = \emptyset $ for $u \neq v$; hence, $|q^{-1}(\tilde{X})|=2|\tilde{X}|$. Now, consider the planar-Rips complex $\RR(X-\tilde{X})$ which is homotopy equivalent to a sphere with $2m$ points deleted, where $|\tilde{X}|=m$. Each removed point $v_i$ from the sphere corresponds to a cycle $\mathcal{C}_i$ in the link of the singular vertex $q(v_i)$ in $\RR(X)$. Figure \ref{pinchedtorus} shows an illustration for $m=1$ case. Further, the cycle $[p(\mathcal{C}_i)]$ is non-trivial in $\pi_1(\Sc(X-\tilde{X}))$ since the projection map $p\colon\RR(X-\tilde{X}) \to \Sc(X-\tilde{X})$ induces a $\pi_1$-isomorphism (Theorem \ref{free}). 

From Proposition \ref{hull}, the convex hulls of any two of the $2m$ cycles in $\RR(X-\tilde{X})$ cannot intersect; hence the shadow $\Sc(X-\tilde{X})$ is homotopy equivalent to a wedge of $2m$ copies of circle $\s^1$. This is possible only if $m=0$ since for $m\geq 1$, the Rips complex $\RR(X-\tilde{X})$ is homotopy equivalent to a wedge of $(2m-1)$ copies of $\s^1$. Therefore, the link of every vertex is connected and the conclusion follows by the virtue of Theorem \ref{normal}.
    
\end{proof}

 Now to extend the result of the above theorem to $n$-dimensional pseudomanifold, we use induction. In order to do so, we first introduce certain prerequisite proximity relations for points in the Euclidean plane.

 \begin{definition}[Antipodal point] Given the boundary complex of $(n+1)$-dimensional cross-polytope, say $\O_n$, the antipodal point of any vertex $v \in \O_n$ is the vertex in $\O_n$ not adjacent to $v$ and is denoted by $\bar{v}$.
 \end{definition}

\begin{lemma} \label{polygon}
    For $X \subset \R^2$ with $(2n+2)$ points($n\geq 1)$, suppose the planar-Rips complex $\RR(X)$ is isomorphic to $\O_{n}$. Then any $x \in X$ is not in the convex hull of the remaining vertices i.e., the shadow $\Sc(X)$ is a simple convex polygon with $2n$ sides.
\end{lemma}
\begin{proof}
    Let if possible, there exists $x \in X$ such that $x \in conv(X-x)$. Then it follows by Caratheodory's theorem that there exists vertices $a,b,c \in (X-x)$ such that $x \in \conv\left\{a,b,c \right\}$. Now there are two cases to consider.
    \begin{enumerate}[label=Case \arabic*:, leftmargin=1.5cm]
        \item[Case 1.] Suppose one of the points $a,b,c$ is an antipodal point of $x$; without loss of generality, assume $a$ to be the antipodal point $\bar{x}$. Since $d(\bar{x},b) < 1$ and $d(\bar{x},c)<1$ in $\RR(X)$, the distance between $\bar{x}$ and any point in $\triangle \bar{x}bc$ is less than 1 i.e., $d(\bar{x},x) < 1$ which is absurd.

        \item[Case 2.] Suppose none of the points $a,b,c$ is the antipodal point $\bar{x}$. Then from Lemma \ref{triangle}, we have $x \in \Lk(\tilde{x})$ i.e., $d(\bar{x},x) <1$ which is again absurd.
\end{enumerate}
\end{proof}

\begin{lemma} \cite[Lemma 6.1]{planar2} \label{guard}
    Let $\mathcal{S}$ be a simple polygon in $\R^2$ and $v$ be a point outside or on $\mathcal{S}$. Suppose further that for every vertex $P$ of $\mathcal{S}$, all the intersections of the ray $\overrightarrow{vP}$ with $\mathcal{S}$ lie between $v$ and $P$. Then there is an edge $AB$ of $\mathcal{S}$ such that for any vertex $P$ of $\mathcal{S}$ the segments $vP$ and $AB$ have non-empty intersection.
\end{lemma}

We refer to such a point $v$ as \textit{guard point}. The motivation to use this terminology stems from the Art Gallery problem (See \cite{art}), a prominent concept in Computational Geometry. In the art gallery problem \cite[p. 8]{artgallery}, a guard point (or sometimes called point guard) refers to a point on or inside a polygon from which the entire interior of the polygon can be observed. It is a location from where a guard can have an unobstructed line of sight to every point inside the gallery. In our case, the guard points are considered to be lying strictly outside the polygon, and the guard points are needed to have every line of sight passing through an edge of the polygon. Furthermore, from Lemma \ref{guard}, it is sufficient for the guard point to have a clear line of sight to only the vertices of the simple polygon.
\begin{definition}[Guard point of a simple polygon along a side] Given a simple polygon $\mathcal{S}$ having an edge $AB$, a point $v$ outside $\mathcal{S}$ is called a $\textbf{guard point}$ of $\mathcal{S}$ alongside $AB$ if for any vertex $P$ of $\mathcal{S}$, the segments $vP$ and $AB$ intersect.
\end{definition}

A fundamental problem in computational geometry is determining the position of a point with respect to a line segment. This involves finding on which side of the line does the point lie, and can be used as a building block to solve more complex problems like line segment intersection and finding the convex hull of a set of points. In summary, these ideas and methods in combinatorial topology for the Euclidean plane case play a fundamental role for effectively analyzing the geometric structures and their properties.

\subsection{Position of a point relative to a line.}
To determine on which side of a given line a point lies, we can use the cross-product of two vectors. Given points $A=(x_1,y_1)$ and $B=(x_2,y_2)$, let $L_{AB}$ represent the line joining them. Let $P=(x,y)$ be a point whose position relative to $L_{AB}$ needs to be determined. Consider the two vectors $\vec{AB}=(x_2-x_1,y_2-y_1)$ and $\vec{AP}=(x-x_1,y-y_1)$, and take their cross-product, $r=\vec{AB}\times\vec{AP}=(x-x_1)(y_2-y_1)-(y-y_1)(x_2-x_1)$. The sign of $r$ indicates the point $P$'s position with respect to line $L_{AB}$: 

\begin{enumerate}[label=(\roman*)]
\item if $r> 0$, then the point is to the left of the line $L_{AB}$; Denote it by $P \in L^+_{AB}$.
\item if $r< 0$, then the point is to the right of the line $L_{AB}$; Denote it by $P \in L^-_{AB}$.
\item if $r= 0$, then the point is on the line $L_{AB}$; Denote it by $P \in L_{AB}$.
\end{enumerate}
The union of line $L_{AB}$ and the region $L^+_{AB}$ will be denoted by $\mathrm{cl}(L^+_{AB})$, which is same as taking the closure of $L^+_{AB}$ in usual topological sense. Note that the order of the points $A$ and $B$ is significant while determining the position of point $P$ relative to the line $L_{AB}$. The orientation of the line is taken to be in the direction from $A$ to $B$. Moreover, it can be easily checked that $L^+_{AB}=L^-_{BA}$.

Additionally, by determining on which side of a line a point lies, we can efficiently test whether a given point lies within the convex hull of three points in the Euclidean plane. For instance, given three non-collinear points $A,B,C$, the region $\mathrm{cl}(L^-_{AB}) \cap \mathrm{cl}(L^-_{BC}) \cap \mathrm{cl}(L^-_{CA}) $ or $\mathrm{cl}(L^+_{AB}) \cap \mathrm{cl}(L^+_{BC}) \cap \mathrm{cl}(L^+_{CA})$ corresponds to the $\triangle ABC$ depending on whether the points $A,B,C$ are in clockwise ($C \in L^-_{AB}$) or counterclockwise ($C \in L^+_{AB})$ order respectively.

For example, the next lemma which we shall use in our subsequent results, is an immediate consequence of the above approach and can be easily verified.

\begin{lemma} \label{chull}
  Let $A, B, C$ be three non-collinear points inside $\R^2$. Consider the set
  \[\mathcal{H}=\left\{X \in \R^2 \mid C \in \conv\left\{A,B,X\right\}\right\}\]
  Then
  \begin{enumerate}
      \item $\mathcal{H}=\mathrm{cl}(L^-_{AC}) \cap \mathrm{cl}(L^+_{BC})$, when $A,B,C$ are in clockwise order 
      \item $\mathcal{H}=\mathrm{cl}(L^+_{AC}) \cap \mathrm{cl}(L^-_{BC})$, when $A,B,C$ are in counterclockwise order 
  \end{enumerate}
\end{lemma}

   \begin{figure}[h!]
    \centering \includegraphics[scale=0.18]{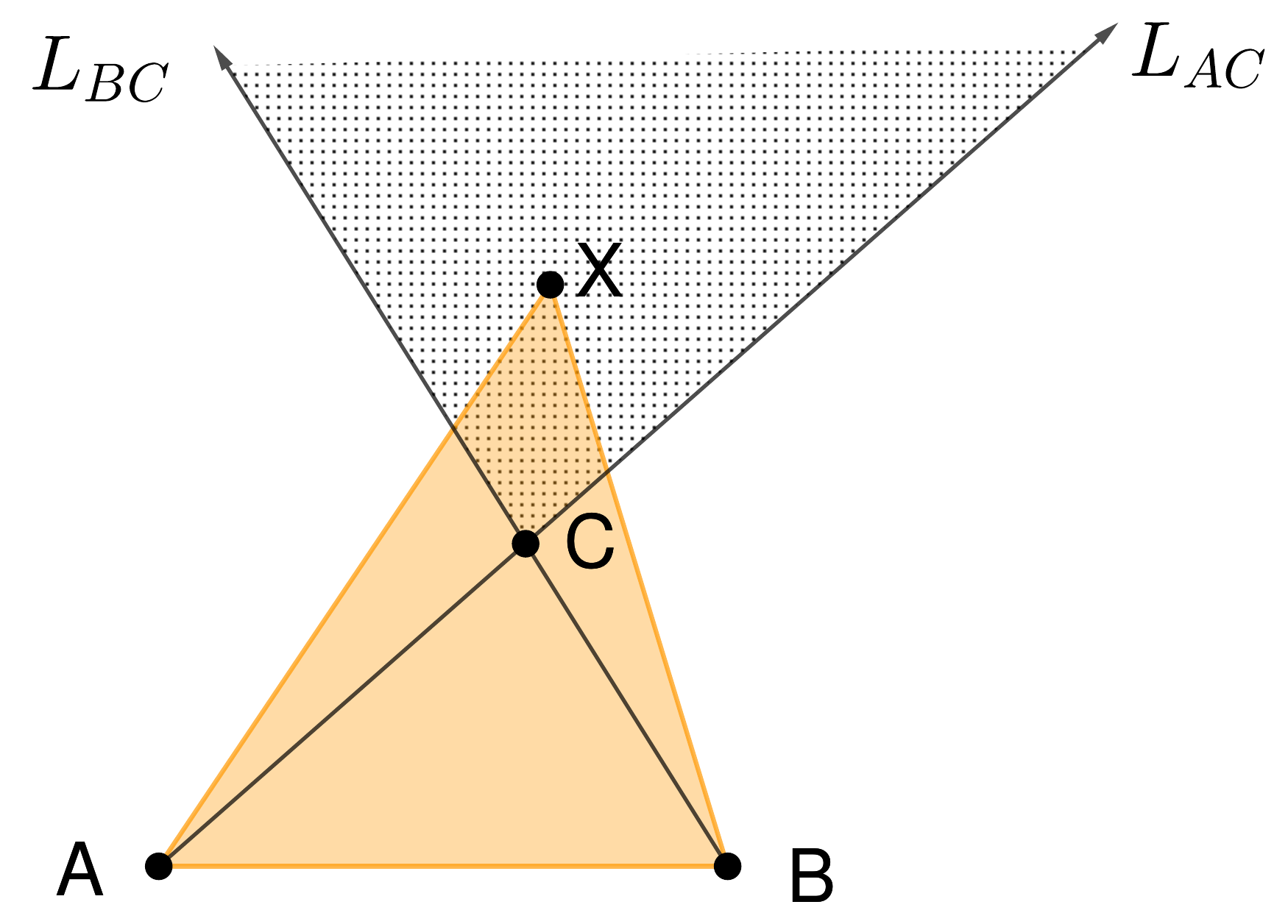}

  \caption{The two cases of Lemma \ref{chull}: The dotted region denotes the set $\mathrm{cl}(L^-_{AC}) \cap \mathrm{cl}(L^+_{BC})$ or $\mathrm{cl}(L^+_{AC}) \cap \mathrm{cl}(L^-_{BC})$, depending on whether points are in clockwise or counterclockwise order respectively.}  
\end{figure}

\subsection{Perpendicular bisector}
Given any two points $A$ and $B$ in $\R^2$, the line of perpendicular bisector of the segment $\overline{AB}$ will be denoted by $L_{A|B}$ and the set of all points whose distance from $A$ is strictly less than distance from $B$ will be denoted by $\mathcal{P}_{A<B}$ i.e., 
 \[\mathcal{P}_{A<B}\coloneq\left\{x \in \R^2 \mid d(A,x) < d(B,x)\right\}\]
Note that the set $\mathcal{P}_{A<B}$ inside $\R^2$ is same as the set of all points that lie on the same side of the line $L_{A|B}$ as the point $A$ does i.e.,  $\mathcal{P}_{A<B}=L^+_{A|B}$. Moreover, it can be easily checked that $L^+_{A|B}=L^-_{B|A}$.

\begin{lemma} \label{convhull}
    Let $A, B , C , D$ be four points in $\R^2$ such that no three points are collinear, the segments $\overline{AB}$ and $\overline{CD}$ intersect, and the distance of $D$ from $C$ is less than the distance from both $A$ and $B$. Suppose $P$ is another point whose distance from $C$ is greater than the distance from both $A$ and $B$, then the point $D$ is in the convex hull of points $A, B, \ \text{and} \ P$ i.e.,  $D \in \conv\left\{ A, B ,P\right\}$.
\end{lemma}

\begin{proof}
\begin{figure}[h!]
    \centering 
    \includegraphics[scale=0.2]{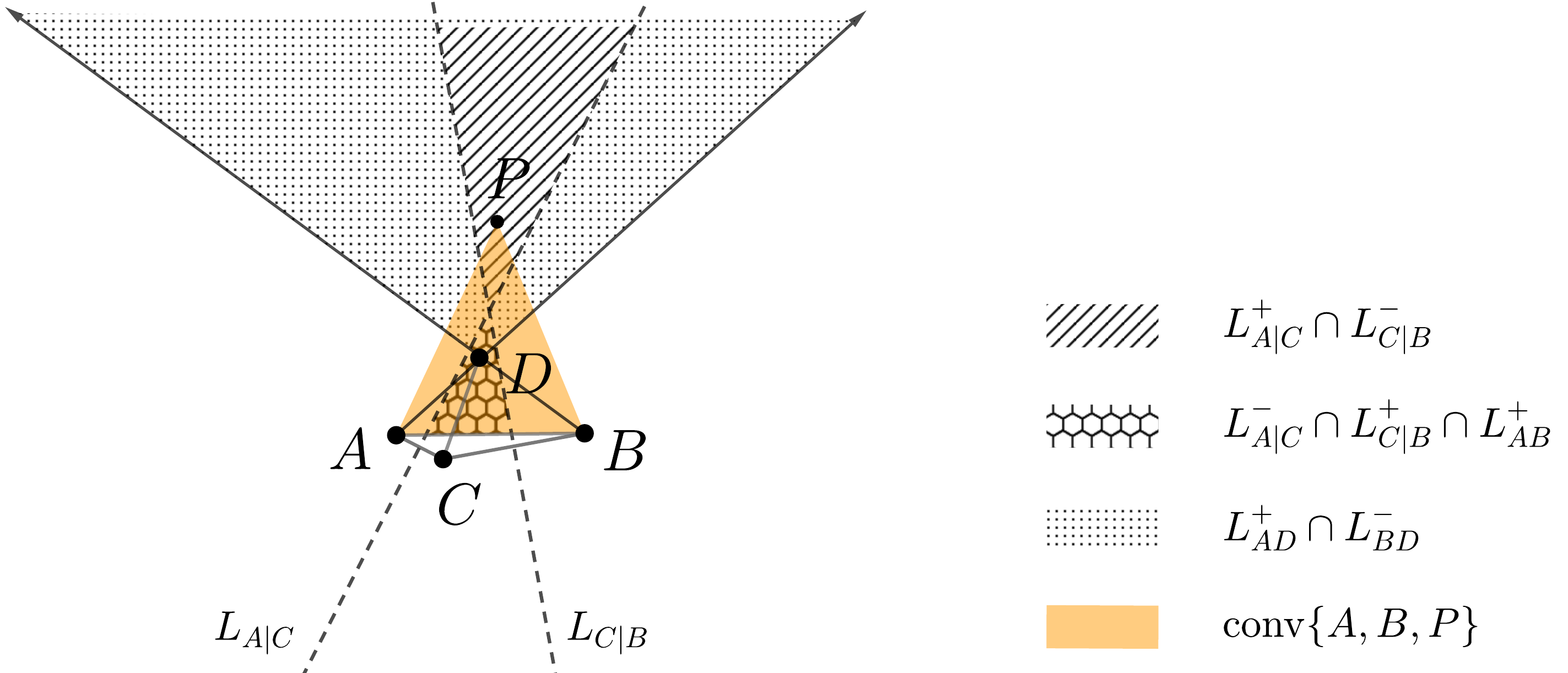}
  \caption{ An illustration of Lemma \ref{convhull}: The textured regions (honeycomb and hatched) denote the respective  possible locations for $D$ and $P$. Correspondingly, $P$ is contained in the region(dotted) where it satisfies $D \in \conv\left\{A,B,P\right\}$. }\label{abcdp}.  
\end{figure}

    Since the segments $\overline{AB}$ and $\overline{CD}$ intersect, and no three points are collinear, the points $C$ and $D$ lie on either side of the line $L_{AB}$; assume without loss of generality, $C \in L^-_{AB}$ and $D \in L^+_{AB}$. Further, the point $D$ being closer to $C$ than both $A$ and $B$ implies $D \in  L^-_{A|C} \cap L^+_{C|B}$. Likewise, $P \in  L^+_{A|C} \cap L^-_{C|B}$. Now it can be easily checked that $ L^+_{A|C} \cap L^-_{C|B} \subset L^+_{AD} \cap L^-_{BD}$ i.e., $P \in L^+_{AD} \cap L^-_{BD}$ (See Figure \ref{abcdp} for an illustration). Consequently, $D \in \conv\left\{ A, B ,P\right\}$ as follows from Lemma \ref{chull}.
    
    \end{proof}

\begin{lemma} \label{quadrilateral}
    Let $\delta > 0$. Suppose $A, ~ B,~ C,~ D,~ X,~ Y$ are six points in $\R^2$ satisfying the following:
    \begin{enumerate}
        \item $ABCD$ is a quadrilateral with each side of length less than $\delta$ and each diagonal of length at least $\delta$.
        \item The distance from $X$ to $C$ and $D$ is less than $\delta$, and the distance from $Y$ to $A$ and $B$ is less than $\delta$.
        \item  The distance between $X$ and $Y$ is at least $\delta$.
    \end{enumerate}
    Then, $X$ is a guard point of $ABCD$ along the side $AB$ if and only if  $Y$ is a guard point of $ABCD$ along the side $CD$. Furthermore, in either case, both points $X$ and $Y$ are less than $\delta$ away from $A$, $B$, $C$, and $D$.
\end{lemma}

\begin{proof}

       \begin{figure}[h!]
       \centering
  \begin{subfigure}[b]{0.40\textwidth}
    \centering \includegraphics[scale=0.25]{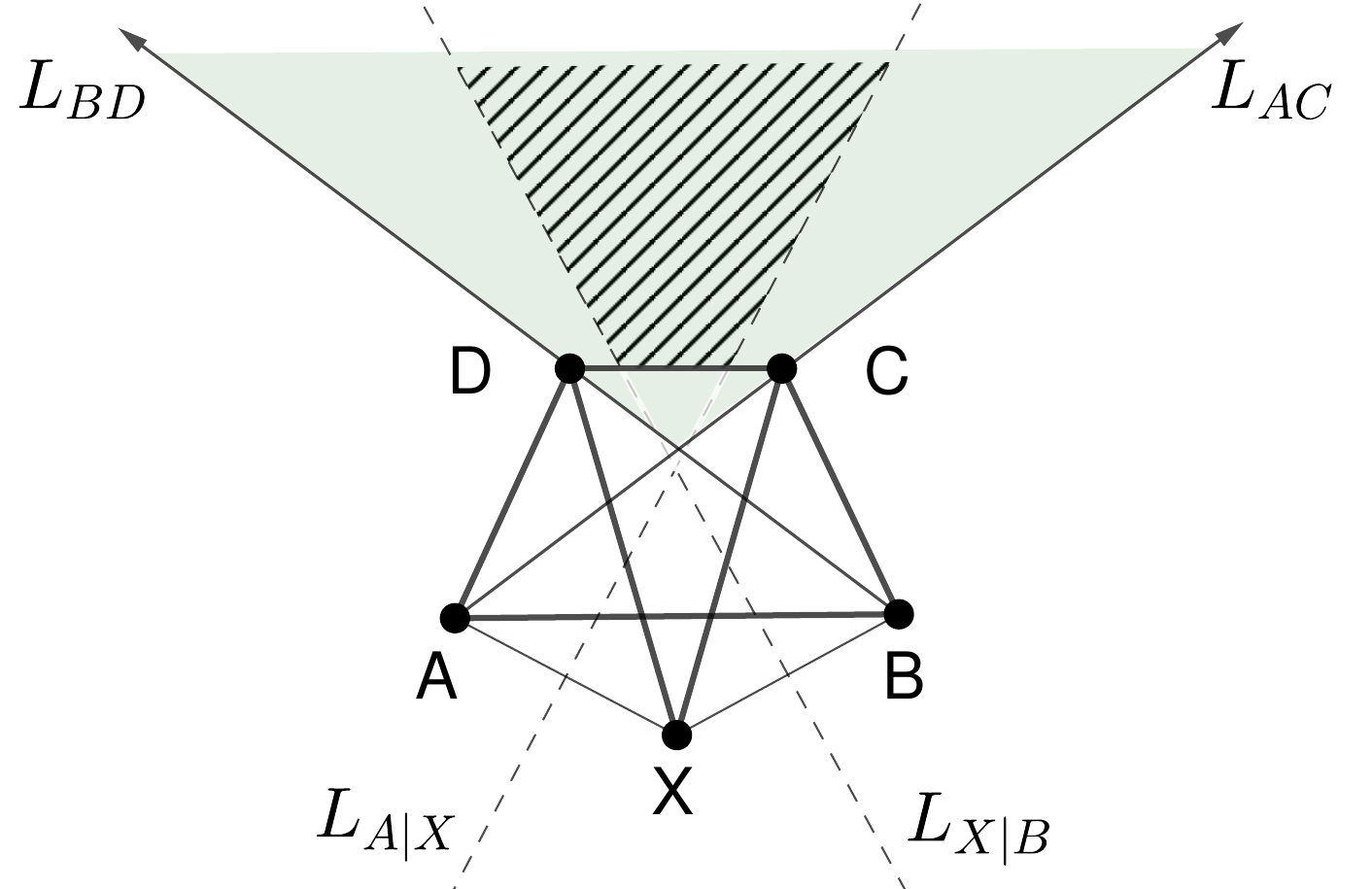}
    \caption{} \label{y1}
  \end{subfigure}
  \begin{subfigure}[b]{0.40\textwidth}
    \centering 
    \includegraphics[scale=0.25]{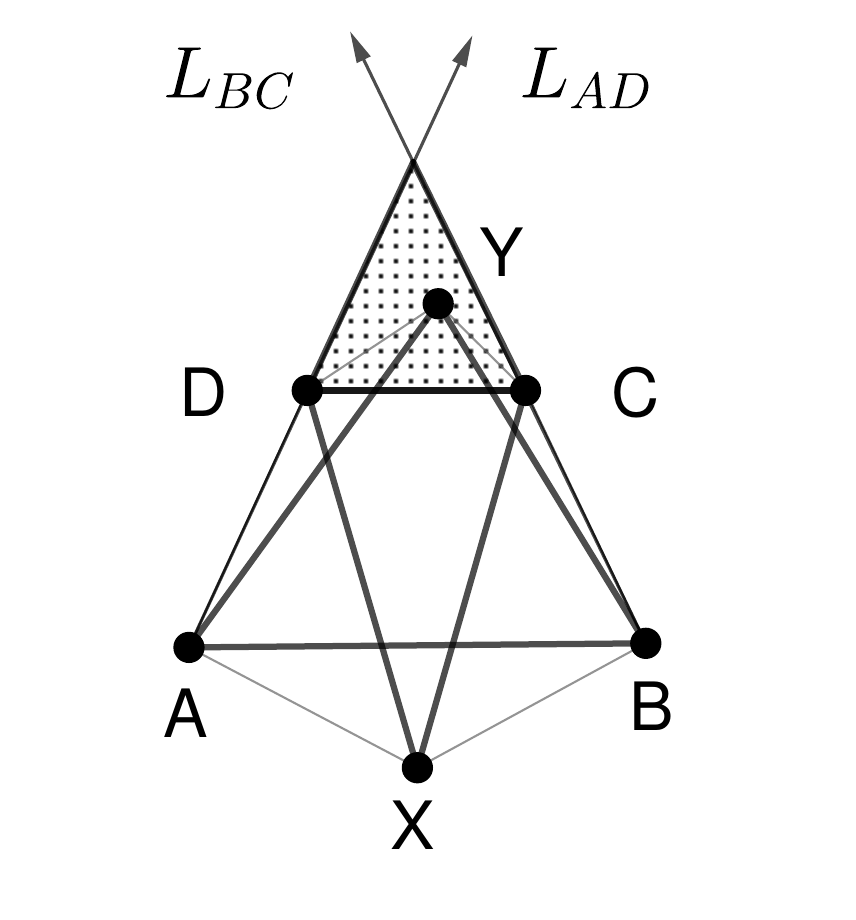}
        \caption{} \label{y2}
  \end{subfigure}
  \caption{An illustration of Lemma \ref{quadrilateral}. In this arrangement, the point $Y$ lies in the intersection of hatched and dotted regions. (Right) The dotted region represents all guard points of $ABCD$ along the side $CD$; the segments $\overline{YA}$ and $\overline{YB}$ intersect $\overline{CD}$.}
  \label{yyy}
\end{figure}

    We prove one direction; the converse follows analogously by swapping $A, ~ B$ with $C, ~D$.
    
    Suppose $X$ is a guard point of $ABCD$ along the side $AB$. From $(1)$, it follows that the quadrilateral $ABCD$ is simple and convex. The points $A$, $B$, and $X$ are not collinear as follows from $(2)$ and $(3)$. Hence, without loss of generality, assume the points $C$ and $D$ are on the left side, and $X$ on the right side of the line $L_{AB}$. The point $Y$ being closer to both $A$ and $B$ than $X$ implies $Y \in  L^+_{A|X} \cap L^-_{X|B}$ (An illustration of a possible arrangement is shown in Figure \ref{yyy}).  We have the following sequential observations:
        \begin{enumerate}[label=(\alph*)]
        \item $C \in L^-_{A|X}$ since the point $C$ is closer to $X$ than $A$. Similarly, $D \in L^+_{X|B}$ since the point $D$ is closer to $X$ than $B$.
        \item This implies $L^+_{A|X} \cap L^-_{X|B} \subset \conv\{A,B,C,D,X\} \cup L^-_{CD}$.
        \item However, $Y \notin \conv\{A,B,C,D,X\}$ since  $d(X,Y) \geq \delta$. Hence, $Y \in L^-_{CD}$.
        \item Additionally, $\mathrm{(a)}$ also implies $L^-_{CD} \cap (L^+_{A|X} \cap L^-_{X|B}) \subset L^+_{AC} \cap L^-_{BD}$ (See Figure \ref{y1}).
        \item Since $d(B,D) \geq \delta$, it follows that $D \notin \conv\left\{A,B,Y\right\}$. Consequently, by Lemma \ref{chull}, we obtain  $Y \in L^-_{AD} \cup L^+_{BD}$. Similarly, we have $Y \in L^-_{AC} \cup L^+_{BC}$.
    \end{enumerate}   
         Henceforth, from the observations $(\mathrm{d})$ and $(\mathrm{e})$, it follows that $Y \in L^-_{CD} \cap L^+_{BC} \cap L^-_{AD}$,  the region of all guard points of $ABCD$ along the side $CD$ (See Figure \ref{y2}). More precisely, $Y$ is a guard point along the side $CD$. 
         
         Furthermore, since the segments $\overline{YA}$ and $\overline{YB}$ intersects the segment $\overline{CD}$, it follows from Lemma \ref{cone} that the distance of $Y$ from points $C$ and $D$ is less than $\delta$. Likewise, the distance of $X$ from points $A$ and $B$ is less than $\delta$.
         
\end{proof}
Note that the planar-Rips complex induced by the six points in the above lemma is nothing but the boundary complex of octahedron $\O_2$.

\begin{theorem} \label{pseudorips}
    For any $n \geq 2$ and $X \subset \R^2$, if the planar-Rips complex $\RR(X)$ is an $n$-dimensional pseudomanifold then $\RR(X)$ is isomorphic to $\mathcal{O}_n$.
\end{theorem}

\begin{proof}
    We proceed by induction. Firstly, for the case when $n=2$, it follows from Theorem \ref{octa} that $\RR(X)$ is isomorphic to $\O_2$. Now let $n>2$ and consider a strongly connected component of link of any  $x \in X$, say $\widetilde{\O}$. By induction, the subcomplex $\widetilde{\O}$ is isomorphic to $\O_{n-1}$. From Lemma \ref{polygon}, the shadow $\Sc(\widetilde{\O})$ is a $2n$-sided convex polygon.

    \begin{customclaim}{1} \label{cc1}
        The point $x$ cannot be inside the polygon $\Sc(\widetilde{\O})$
    \end{customclaim} 
    \begin{proof}\renewcommand{\qedsymbol}{}
        On contrary, suppose $x$ is inside $\Sc(\widetilde{\O})$. Then there exists an $(n-1)$-dimensional simplex $\Delta$ in $\widetilde{\O}$ such that $x \in \conv((\Delta))$ . Since $\RR(X)$ is an $n$-dimensional pseudomanifold, the $(n-1)$-dimensional simplex $\Delta$ is shared by exactly two $n$-dimensional simplices, one of which is $\Delta\cup\left\{ x \right\}$. Hence, there exists another point $y$ outside $\widetilde{\O}$ such that the vertices of $\Delta$ are adjacent to $y$ i.e., $\Delta\cup\left\{ y \right\}$ is an $n$-dimensional simplex in $\RR(X)$. Now from Lemma \ref{triangle}, we have $x \in \Lk(y)$ i.e., $\Delta\cup\left\{ x \right\} \cup \left\{ y \right\}$ is an $(n+1)$-dimensional simplex. But this is not possible since $\RR(X)$ is $n$-dimensional. This concludes the proof of Claim \ref{cc1}.
       
    \end{proof}
    Hence the point $x$ is outside the polygon $\Sc(\widetilde{\O})$. Now we show that $x$ is a guard point of polygon $\Sc(\widetilde{\O})$ along a side.

    \begin{customclaim}{2} \label{cc2}
        There exists a side $ab$ of the polygon $\Sc(\widetilde{\O})$ such that $x$ is a guard point of $\Sc(\widetilde{\O})$ along the side  $ab$
    \end{customclaim} 
    \begin{proof}\renewcommand{\qedsymbol}{}
    \begin{figure}[h!]
    \centering \includegraphics[scale=0.18]{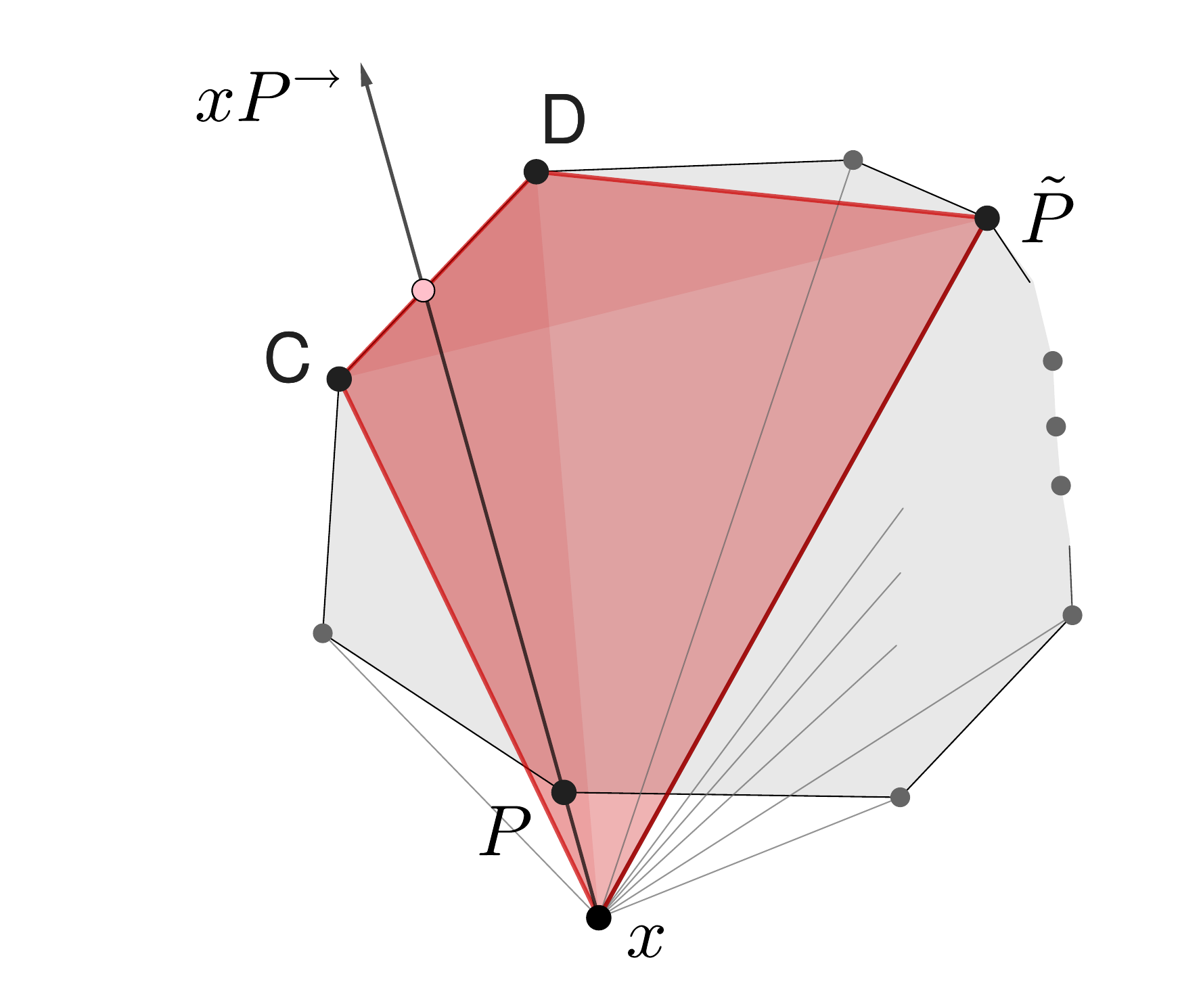}

  \caption{An illustration of the case when the ray $\overrightarrow{xP}$ intersects the side $cd$, and the intersection is outside the segment $\overline{xP}$. This yields a $3$-simplex $\left\{x,\tilde{P},c,d \right\}$ containing the point $P$. }\label{xpray}.  
\end{figure}
        Following Lemma \ref{guard}, it suffices to show that for any vertex $P$ in the polygon $\Sc(\widetilde{\O})$, the ray $\overrightarrow{xP}$ does not intersect $\Sc(\widetilde{\O})$ outside the segment $\overline{xP}$. On contrary, suppose there exists a vertex $P$ such that the extension of ray $\overrightarrow{xP}$ intersects the polygon $\Sc(\widetilde{\O})$, say at some edge $cd$ (See Figure \ref{xpray}). Then we have $P \in \conv\left\{x,c,d\right\}$. Now consider the antipodal point $\tilde{P}$ of $P$ in $\widetilde{\O}$ where the points $x,c,d$ are adjacent to $\tilde{P}$ i.e., $\left\{x,c,d\right\} \subset \Lk(\tilde{P})$. Hence from Lemma \ref{triangle}, we have $P \in \Lk(\tilde{P})$ i.e., the points $P$ and $\tilde{P}$ are adjacent which is a contradiction. This concludes the proof of Claim \ref{cc2}.
        
    \end{proof}
    Now consider any $(n-1)$-dimensional simplex $\Delta$ in $\widetilde{\O}$  which contains the vertices $a$ and $b$. Since $\RR(X)$ is an $n$-dimensional pseudomanifold, the $(n-1)$-dimensional simplex $\Delta$ is shared by exactly two $n$-dimensional simplices, one of which is $\Delta\cup\left\{ x \right\}$. Hence, there exists another point $y$ outside $\widetilde{\O}$ such that the vertices of $\Delta$ are adjacent to $y$ i.e., $\Delta\cup\left\{ y \right\}$ is an $n$-dimensional simplex in $\RR(X)$. Note that $x$ and $y$ are non-adjacent since otherwise $\Delta \cup \left\{x,y\right\}$ forms an $(n+1)$-simplex.

    \begin{customclaim}{3} \label{cc3}
        The point $y$ is adjacent to all the vertices of $\widetilde{\O}$
    \end{customclaim}
    \begin{proof}\renewcommand{\qedsymbol}{}
       \begin{figure}[h!]
       \centering
  \begin{subfigure}[b]{0.40\textwidth}
    \centering \includegraphics[scale=0.15]{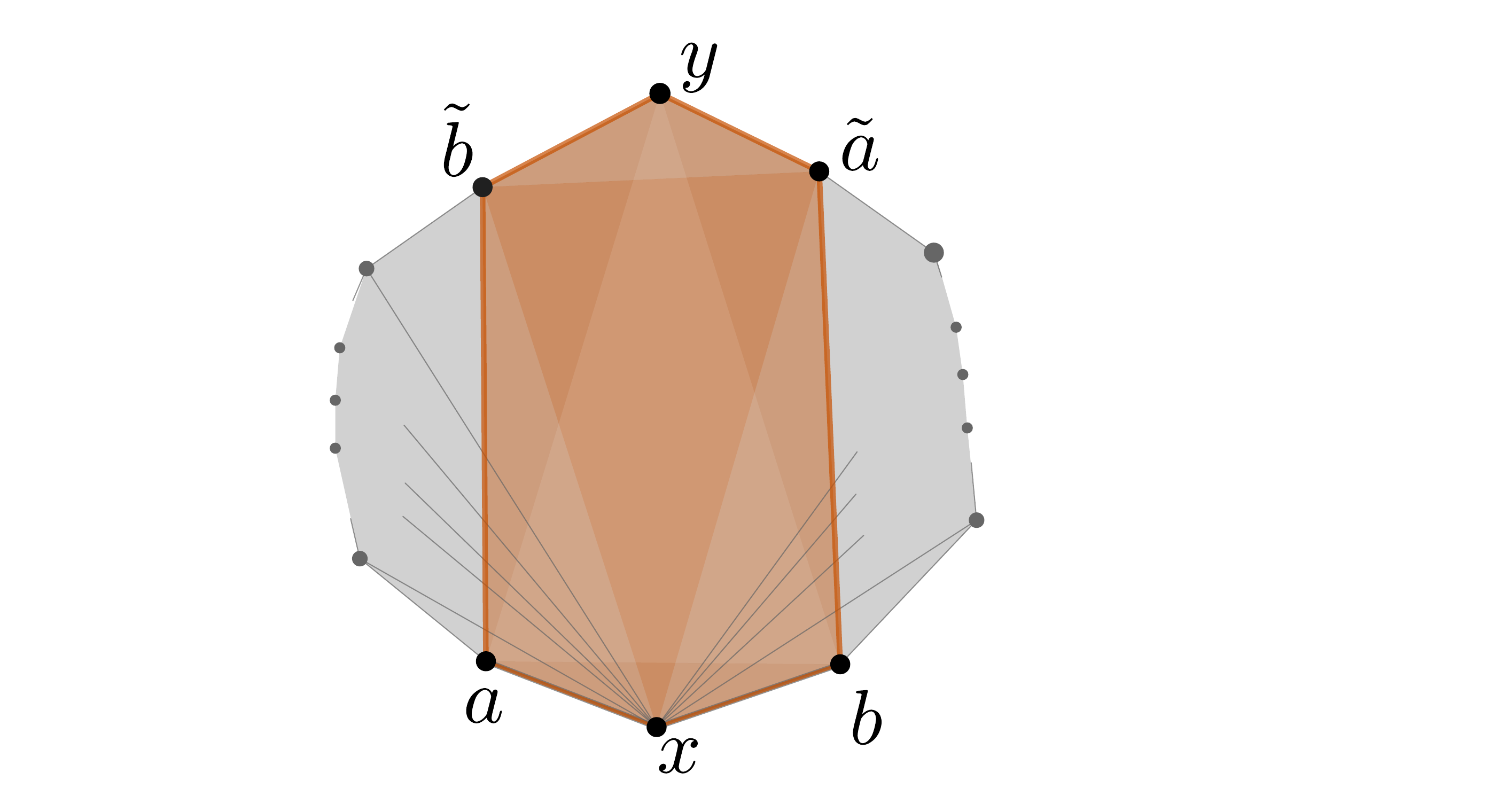}
    \caption{} \label{abcdxy}
  \end{subfigure}
  \begin{subfigure}[b]{0.40\textwidth}
    \centering 
    \includegraphics[scale=0.15]{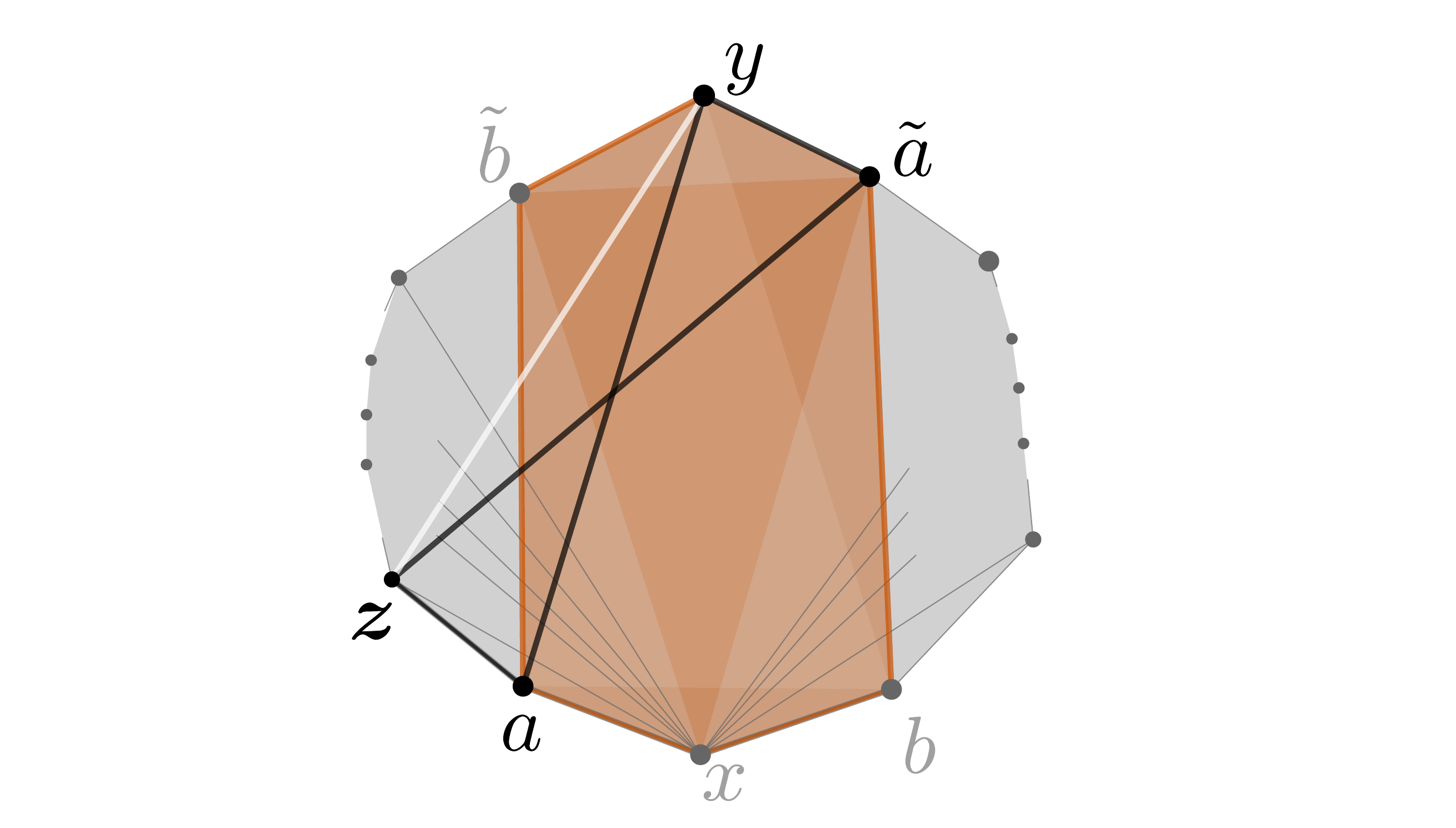}
        \caption{} \label{yz_adj2}
  \end{subfigure}
  \caption{The points $x$ and $y$ are guard points of polygon $\tilde{\O}$. The adjacency of $y$ with $\tilde{a}$ \& $\tilde{b}$, and $y$ being a guard point of polygon $\tilde{\O}$ along the side $\overline{\tilde{a}\tilde{b}}$ prompts $y$ to be adjacent to every other vertex of $\tilde{\O}$. On the right, the segments $\overline{ya}$ and $\overline{z\tilde{a}}$ intersect, implying adjacency between $y$ and $z$.} 
\end{figure}
    Firstly, the point $y$ is adjacent to $a$ and $b$. Now consider the antipodal points $\tilde{a}$ and $\tilde{b}$ of $a$ and $b$ respectively. With the points $a,b,\tilde{a},\tilde{b},x,y$, we are in the setting of Lemma \ref{quadrilateral} with the parameter $\delta=1$, and the quadrilateral being $ab\tilde{a}\tilde{b}$. Hence from Lemma \ref{quadrilateral}, the point $y$ is adjacent to $\tilde{a}$ and $\tilde{b}$, and is a guard point of quadrilateral $ab\tilde{a}\tilde{b}$ along the side $\tilde{a}\tilde{b}$. The induced structure is illustrated in Figure \ref{abcdxy}. 
    
    Now let $z$ be any point of $\widetilde{\O}$ other than $a,b,\tilde{a},\tilde{b}$. Since $y$ is a guard point of the quadrilateral $ab\tilde{a}\tilde{b}$, either the segment $\overline{ya}$ intersects the segment $\overline{z\tilde{a}}$ or the segment $\overline{yb}$ intersects the segment $\overline{z\tilde{b}}$ in $\Sc(X)$, depending on whether the point is on the left side or the right side of line $L_{xy}$ (For instance, see Figure \ref{yz_adj2}). For both the cases, from Lemma \ref{cone}, either $y$ or $z$ is an apex since the antipodal points $a$ and $\tilde{a}$ or $b$ and $\tilde{b}$ cannot be adjacent i.e., $y$ is adjacent to $z$. This concludes the proof of Claim \ref{cc3}.   
    \end{proof}

Hence, the subcomplex spanned by $x$, $y$ and vertices of $\widetilde{\O}$ is isomorphic to $\O_n$. Since the only subcomplex of a pseudomanifold which is also a pseudomanifold of same dimension is itself, the Rips complex $\RR(X)$ is isomorphic to $\O_n$.

\end{proof}

The above characterization of pseudomanifolds provides further  insight into the Rips complexes having a weak-pseudomanifold structure, where each strongly connected component is a pseudomanifold, and the corresponding induced subcomplex again inherits the planar-Rips structure. 

\begin{definition}[Iterated chain of $\O_n$s] \label{iterated} A finite $n$-dimensional simplicial complex $(n \geq 2)$ is said to be an \textbf{iterated chain of $\O_n$s} if it can be expressed as the union $\bigcup_{i=1}^{m} P_i$, where each $P_i$ is isomorphic to $\O_n$, no three of them have a non-empty pairwise intersection, and any two of them with non-empty intersection intersect
     \begin{enumerate}[label=(\alph*)]
         \item at a single vertex when $n=2$.
         \item either at a single vertex or a single edge when $n>2$.
     \end{enumerate}
\end{definition}

It can be easily checked that an iterated chain of $\O_n$s is a weak-pseudomanifold, where each $P_i$ is a strongly connected component.

\begin{proposition} \label{iteratedwedge}
    Let $K$ be an iterated chain of $\O_n$s, then $K$ is homotopy equivalent to wedge sum of circles and $n$-spheres i.e., $K \simeq \bigvee^m{\s^n} \vee \bigvee^p{\s^1}$ for some $p \geq 0$, where $m=\#$ of strongly connected components in $K$.
\end{proposition}
\begin{proof}
    Any two copies of $\O_n$ can share at most a single vertex or a single edge. Further, any two $1$-simplices with a common vertex cannot be a part of three distinct copies of $\O_n$. This allows for the contraction of edges shared by two copies of $\O_n$, to a single point upto homotopy, resulting in a chain of $n$-spheres $\s^n$ along the vertices. However, these spheres can be arranged in a cyclic fashion, which gives rise to copies of circles in its homotopy type. 
\end{proof}
\begin{theorem} \label{weakiswedge}
     For any $n \geq 2$ and a finite subset $X \subset \R^2$, if the planar-Rips complex $\RR(X)$ is an $n$-dimensional weak-pseudomanifold then $\RR(X)$ is isomorphic to an iterated chain of $\O_n$s.
    
\end{theorem}
\begin{proof}
  Firstly, each strongly connected component is isomorphic to $\O_n$ since the corresponding induced subcomplex on each component inherits a planar-Rips structure; denote the strongly connected components by $\{P_i\}_{i=1}^m$ (See Figure \ref{wedgeiter} for an illustration). Now, suppose any two strongly connected components, say $P_i$ and $P_j$ have a non-empty intersection, then their intersection is a subcomplex of codimension at least $2$. Further, from Lemma \ref{polygon}, any $k$-simplex ($k \geq 2$) in $P_i$ and $P_j$ is non-degenerate, and the polygons $\Sc(P_i)$ and $\Sc(P_j)$ are convex; hence the intersection $P_i \cap P_j$ is a vertex when $n=2$, and is either a vertex or an edge when $n>2$.
  
  Now from Lemma \ref{k16}, no more than two strongly connected components can intersect at a single vertex. Finally, it again follows from the convexity of $\Sc(\O_n)$ that any three $P_i$s cannot intersect pairwise inside $\R^2$.
\end{proof}

\begin{corollary} \label{weakhomotopy}
    If $\RR(X)$ is an $n$-dimensional weak-pseudomanifold, then 
    \[\RR(X)\simeq \bigvee_{i=1}^m{\s^n} \vee \bigvee_{j=1}^p{\s^1} \ ,\]
    for some $p \geq 0$, where $m=\#$ of strongly connected components in $\RR(X)$.
\end{corollary}
\begin{remark}
Note that $p = 0$ unless a subcollection of strongly connected components can be arranged in a ``cyclic" manner, in which case $p=1$.  Figure \ref{cyclewedge} illustrates how such a cyclic structure would look; however, it remains to be determined whether such a structure is possible. Additionally, if $p \neq 0$ then $m > 3$.
           \begin{figure}[h!]
       \centering
  \begin{subfigure}[b]{0.40\textwidth}
    \centering \includegraphics[scale=0.2]{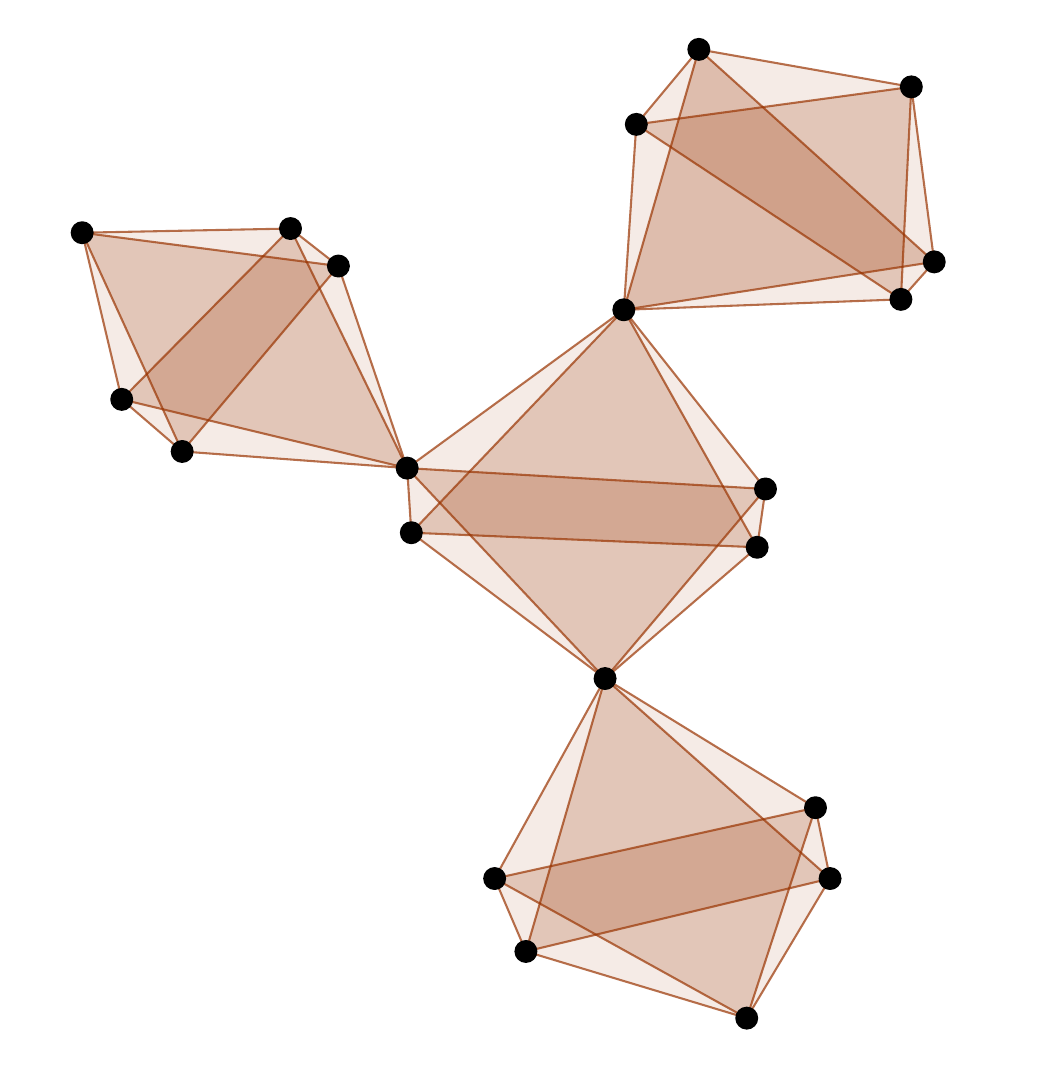}
    \caption{$m=4$, $p=0$} \label{iterwedge}
  \end{subfigure}
  \begin{subfigure}[b]{0.40\textwidth}
    \centering 
    \includegraphics[scale=0.18]{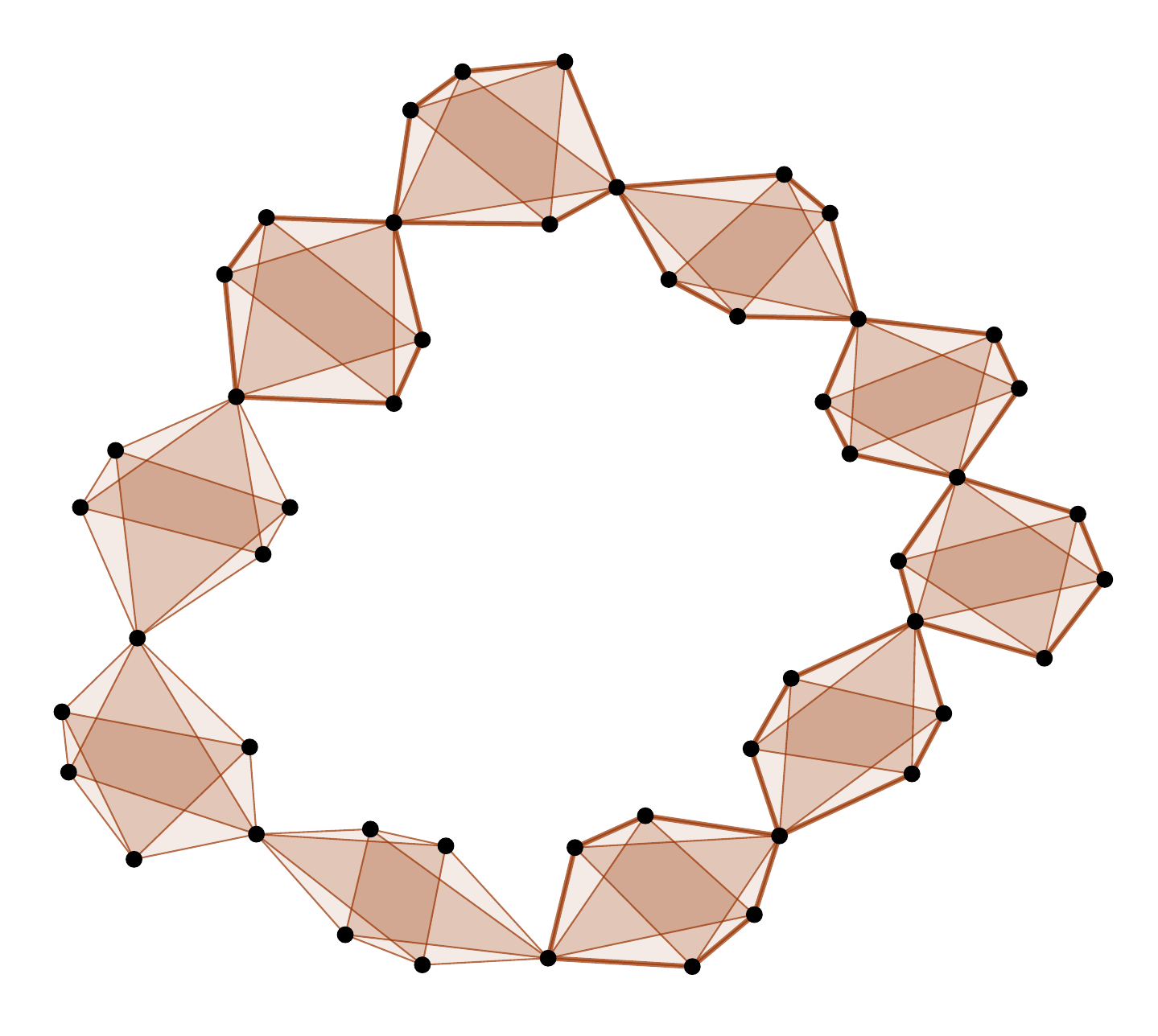}
        \caption{$m=10$, $p=1$} \label{cyclewedge}
  \end{subfigure}
  \caption{Illustration of iterated chain of $\O_n$s for $n=2$. Correspondingly, the values of $m$ and $p$ are specified as in Corollary \ref{weakhomotopy}.} 
  \label{wedgeiter}
\end{figure}
\end{remark}

 \section{Characterization of minimal $2$-cycles and two-dimensional closed planar-Rips structures} \label{sec4}

In this section, we will prove Theorem \ref{C}, laying the groundwork for recognizing two-dimensional, pure and closed planar-Rips structures, a more general class of $2$-pseudomanifolds. We will achieve this by first establishing its inherent local structure and showing that a minimal 2-cycle admitting a planar-Rips structure is isomorphic to $\O_2$.

\begin{definition} 
\begin{enumerate}[label=(\alph*)]
    \item A finite $n$-dimensional pure simplicial complex is said to be closed if every simplex of codimension $1$ has degree at least $2$ i.e., every $(n-1)$-simplex is contained in at least two $n$-simplices.
    \item A finite $n$-dimensional pure simplicial complex is said to be a \textbf{minimal $n$-cycle} if the only subcomplex having $n^{th}$ homology of rank $1$ is itself.
\end{enumerate}
\end{definition}
One can easily verify that minimal $n$-cycles are strongly connected and closed.

\begin{lemma} \label{degenerate}
    For $X \subset \R^2$, if the planar-Rips complex $\RR(X)$ is an $n$-dimensional pure and closed simplicial complex ($n\geq 2$), then every simplex is non-degenerate, meaning that no vertex in a simplex is contained within the convex hull of its remaining vertices in the shadow complex $\Sc(X)$.
\end{lemma}

\begin{proof}
    Assume the contrary i.e., let there be a degenerate simplex, say $\sigma$ of dimension at least $2$. Then there exists a vertex $v$ which is contained in the convex hull of remaining vertices of $\sigma$. Let $\tau$ be a maximal simplex containing $\sigma$ then $v \in \conv((\tau) - v)$. Since $\RR(X)$ is closed, there exists another $n$-simplex $\gamma$ that does not contain $v$, such that $\gamma \neq \tau$ and $\tau -v \subset \gamma$. Now from Lemma \ref{triangle}, $v \cup \gamma$ is an $(n+1)$-simplex in $\RR(X)$ which is not possible.
\end{proof}

We will use the above proposition specifically for the case when $n=2$. This proposition also implies that for a pure and closed two-dimensional planar-Rips complex, the projection of each $n$-simplex in the shadow $\Sc(X)$ is an $(n+1)$-sided convex polygon. Consider all those $1$-simplices in $\RR(X)$, whose projected interiors in the shadow $\Sc(X)$ intersect the boundary of $\Sc(X)$. These $1$-simplices will be referred to as \textit{boundary edges}.

\begin{definition}[$\gamma_2$-configuration]
   A \textbf{$\pmb{\gamma_2}$-configuration} is an induced subcomplex of $\RR(X)$ spanned by four vertices, which contains precisely two $2$-simplices, and the edge shared by them lies on one side of the line joining the two non-adjacent vertices in the shadow $\Sc(X)$.
    \begin{figure}[h!]
    \centering \includegraphics[scale=0.45]{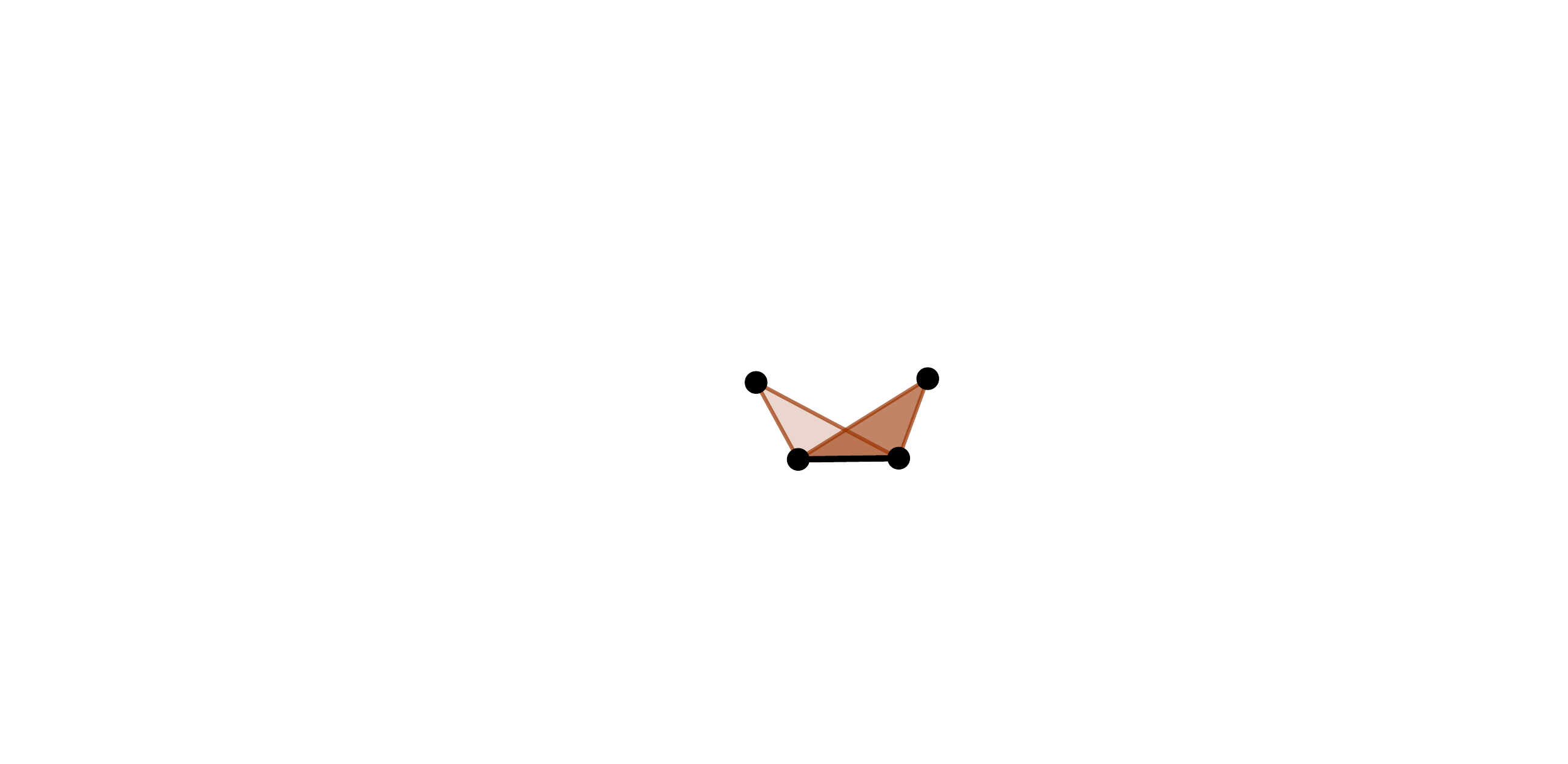} 
  \caption{An illustration of $\gamma_2$-configuration; The highlighted edge (in black) denotes the base of the configuration.} \label{gamma2}
\end{figure}
\end{definition}
Additionally, the edge($1$-simplex) shared between two $2$-simplices in the $\gamma_2$-configuration will be referred to as the \textbf{\textit{base}} of the configuration. The following result is a simple consequence of Lemma \ref{degenerate}.

\begin{proposition} 
In a finite two-dimensional, pure and closed planar-Rips complex $\RR(X)$, 

 \begin{thmenum}

    \item Every boundary edge is a base of some $\gamma_2$-configuration. \label{bedge1}
    \item If two edges intersect at their projected interiors, then they induce a $\gamma_2$-configuration.\label{bedge2}
     
 \end{thmenum} \label{bedge} 
\end{proposition} 
\begin{proof}
    Since $\RR(X)$ is closed, consider any two $2$-simplices containing the boundary edge; their projected interiors intersect in the shadow $\Sc(X)$. There are two possibilities: either one of the two simplices is degenerate or they together form a $\gamma_2$-configuration. However, from Lemma \ref{degenerate}, both the simplices are non-degenerate. This proves $(i)$.

       \begin{figure}[h!]
       \centering
  \begin{subfigure}[b]{0.25\textwidth}
    \centering \includegraphics[scale=0.25]{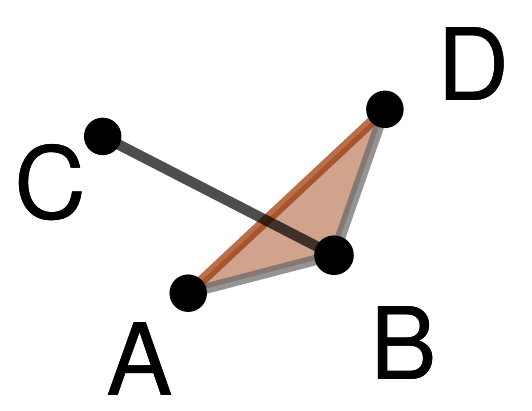}
    \caption{} \label{abcd_gamma}
  \end{subfigure}
  \begin{subfigure}[b]{0.25\textwidth}
    \centering 
    \includegraphics[scale=0.25]{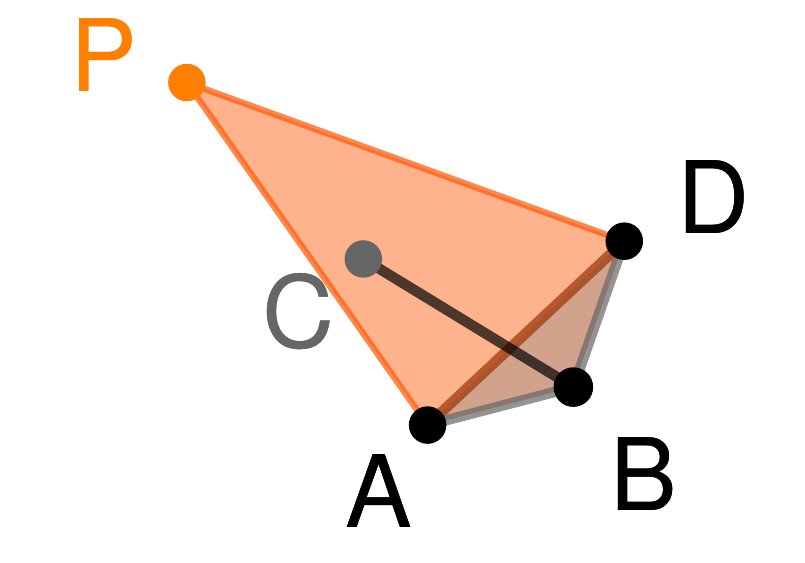}
        \caption{} \label{forbid_gamma}
  \end{subfigure}
    \begin{subfigure}[b]{0.25\textwidth}
    \centering 
    \includegraphics[scale=0.25]{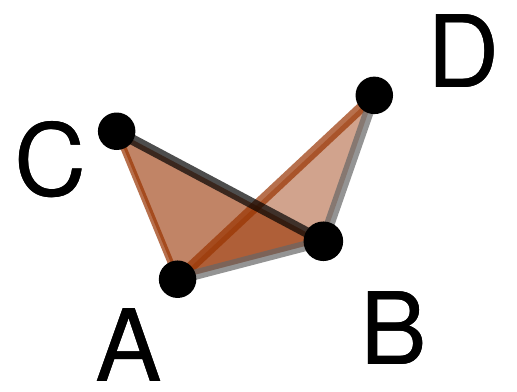}
        \caption{} \label{abcd_gamma1}
  \end{subfigure}
  \caption{Illustration of all possible and forbidden cases in the proof of Proposition \ref{bedge2}.} 
\end{figure}

    For (ii), suppose $[AD]$ and $[BC]$ are two edges in $\RR(X)$ whose projected interiors intersect. From Lemma \ref{cone}, assume without loss of generality that the vertex $B$ is adjacent to both $A$ and $D$ as illustrated in Figure \ref{abcd_gamma}. Now it suffices to show that the vertex $C$ is either adjacent to $A$ or $D$. Note that $C$ cannot be adjacent to both $A$ and $D$ since $\RR(X)$ is two-dimensional. Further from Lemma \ref{degenerate}, no three vertices of $A,B,C,D$ are collinear.
    
    On contrary, suppose $C$ is neither adjacent to $A$ nor $D$. Since $\RR(X)$ is closed, a vertex $P \in X$ exists which is adjacent to both $A$ and $D$, but not adjacent to $B$. It now follows from Lemma \ref{convhull} that $C \in \conv\left\{A,D,P\right\}$, as shown in Figure \ref{forbid_gamma}. Consequently, $C$ is adjacent to both $A$ and $D$ which is absurd.
    
\end{proof}    

\begin{corollary}
    Every finite two-dimensional, pure and closed planar-Rips complex $\RR(X)$ contains a $\gamma_2$-configuration.
\end{corollary}
Note that the above proposition is an extended version of Lemma \ref{cone} for the case when $\RR(X)$ has a two-dimensional, pure and closed simplicial structure. Here, any two segments when they intersect, precisely yield a $\gamma_2$-configuration.
\begin{lemma} \label{obtuse}
    Let $\RR=\sub{ABCD}$ be a $ \gamma_2$-configuration containing the simplices $[ABC]$ and $[ABD]$ such that the points $C$ and $D$ are non-adjacent, and the segments $\overline{AD}$ and $\overline{BC}$ intersect. Then $\phase{CAB}+ \phase{ABD}> 180^\circ$ i.e.,  at least one of   $\phase{CAB}$ or $\phase{ABD}$ is obtuse.
\end{lemma}
\begin{proof}
    This simply follows with the argument that the length of side $\overline{CD}$ is larger than the diagonal lengths $\overline{AD}$ and $\overline{BC}$ in the quadrilateral $ABDC$.
\end{proof}
Now we show that every closed two-dimensional planar-Rips complex has a copy of $\O_2$ as an induced subcomplex. The strategy employed is to use the `closed'ness property of $\RR(X)$ to add two more vertices along with the $\gamma_2$-configuration, and then use the proximity relations on these six vertices to establish the precise adjacency, proving that the induced subcomplex on these six vertices is isomorphic to $\O_2$.  
\begin{theorem} \label{minimal}
    For any $X \subset \R^2$, if the planar-Rips complex $\RR(X)$ is a two-dimensional pure and closed simplicial complex then $\RR(X)$ has an induced subcomplex isomorphic to $\O_2$. Hence, the second Betti number $b_2(\RR(X))$ is non-zero.
\end{theorem}
\begin{proof}
   We will repeatedly use Proposition \ref{bedge2} to prove this theorem. Firstly, as follows from this proposition, we start by taking a $ \gamma_2$-configuration, say $\sub{ABCD}$ as an induced subcomplex in $\RR(X)$ where the simplices $[ABC]$ and $[ABD]$ are in $\RR(X)$, the points $C$ and $D$ are non-adjacent, and the segments $\overline{AD}$ and $\overline{BC}$ intersect. Note that the points $C$ and $D$ are on the same side of the line $L_{AB}$ i.e., without loss of generality, assume $C,D \in L^+_{AB} $. 
   
   Now, since $\RR(X)$ is closed, there exists at least another point for each of $[AD]$ and $[BC]$, say $P$ and $Q$ respectively, such that the two simplices $[ADP]$ and $[BCQ]$ are in $\RR(X)$. Note that $P\neq Q$  and the points $A$ and $Q$ are non-adjacent since $\RR(X)$ is two-dimensional. Likewise, the points $B$ and $P$ are non-adjacent. Hence, the point $P$ lies on the same side of the line of perpendicular bisector $L_{A|B}$ as $A$ does i.e., $P \in L^+_{A|B} $ and the point $Q$ lies on the same side of $L_{A|B}$ as $B$ does i.e., $Q \in L^-_{A|B} $.
   Using Lemma \ref{obtuse}, we also assume without loss of generality that $\phase{CAB}$ is obtuse.
    
\begin{customclaim}{1} \label{c1}
    The points $C$ and $D$ lie on either side of the bisector $L_{A|B}$.
\end{customclaim}
\begin{proof}\renewcommand{\qedsymbol}{}
        \begin{figure}[h!]
    \centering \includegraphics[scale=0.15]{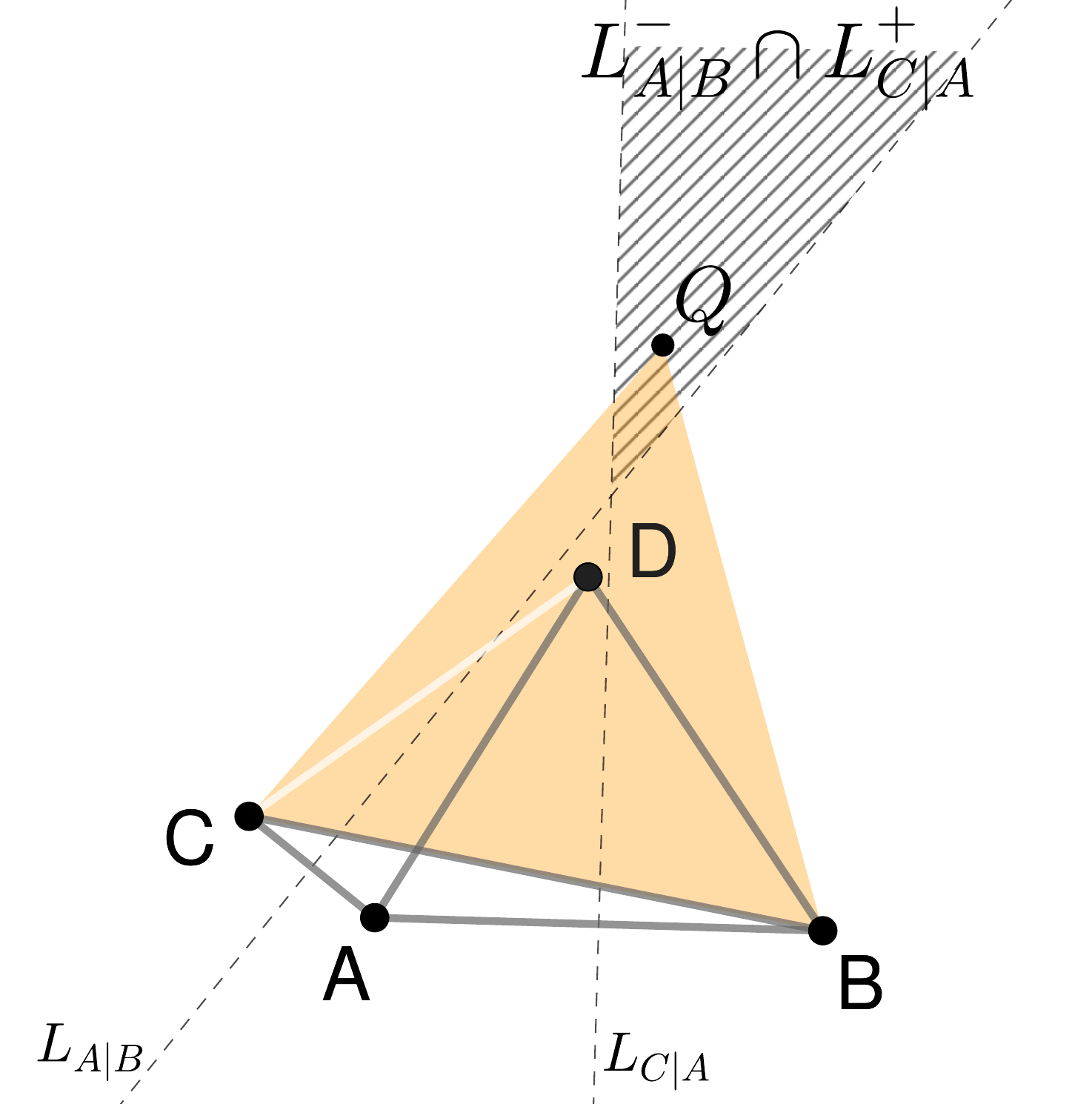}

  \caption{An illustration of Claim \ref{c1}: Here, the points $C$ and $D$ are taken to be on the left side of the bisector $L_{A|B}$. Corresponding to this arrangement, $Q$ lies in the hatched region, implying that $D \in \conv\left\{B,C,Q\right\}$, as represented by the solid shaded region.}\label{same_side}.  
\end{figure}
    Let if possible, both $C$ and $D$ be on the same side of line $L_{A|B}$ i.e., let $C,D \in L^+_{A|B} $(since $\phase{CAB}$ is obtuse). Consider the set $L^-_{A|B} \cap L^+_{C|A} $; the point $Q$ being closer to both $B$ and $C$ than $A$ implies $Q \in L^-_{A|B} \cap L^+_{C|A} $. This is illustrated in Figure \ref{same_side}. Hence from Lemma \ref{convhull}, $D \in \conv\left\{B,C,Q\right\}$ i.e.,  $C$ and $D$ are adjacent which is not true. This concludes the proof of Claim \ref{c1}.
   
\end{proof}
Hence, $C \in L^+_{A|B}$ and $D \in L^-_{A|B}$.  With the above inference about points $C$ and $D$, it is possible to draw further conclusion about the specific region in which points $P$ and $Q$ should lie.

   \begin{customclaim}{2} \label{c2}
       $P, \ Q \in L^+_{AD} \cap L^-_{BC}$.
   \end{customclaim}
\begin{proof} \renewcommand{\qedsymbol}{}
 \begin{figure}[h!]
 \centering
  \begin{subfigure}[b]{0.45\textwidth}
    \centering \includegraphics[scale=0.16]{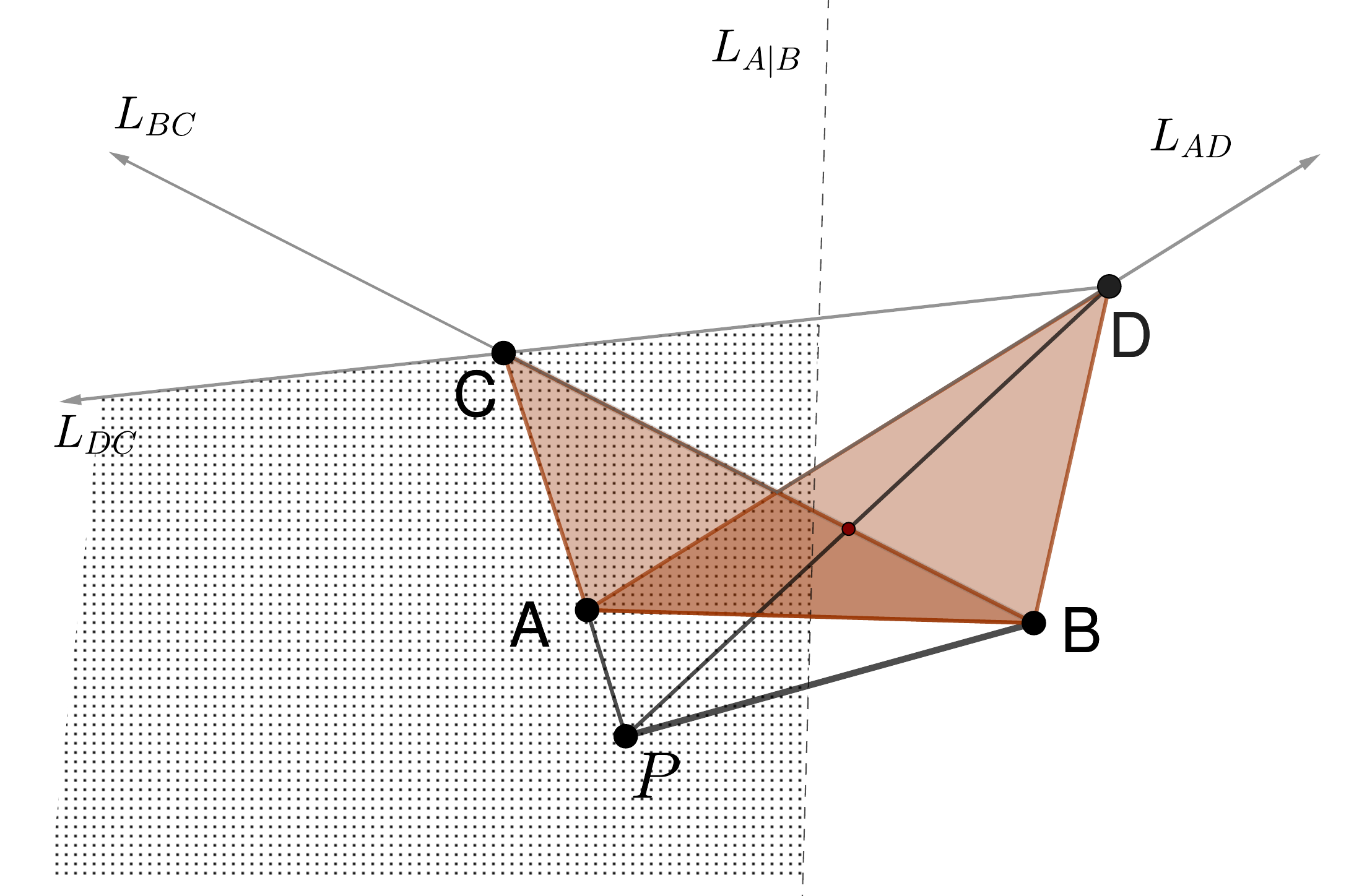}
    \caption{$P \in L^+_{A|B} \cap \ \mathrm{cl}(L^+_{BC}) \cap L^-_{DC}$} \label{m1}
  \end{subfigure}
  \begin{subfigure}[b]{0.45\textwidth}
    \centering 
    \includegraphics[scale=0.14]{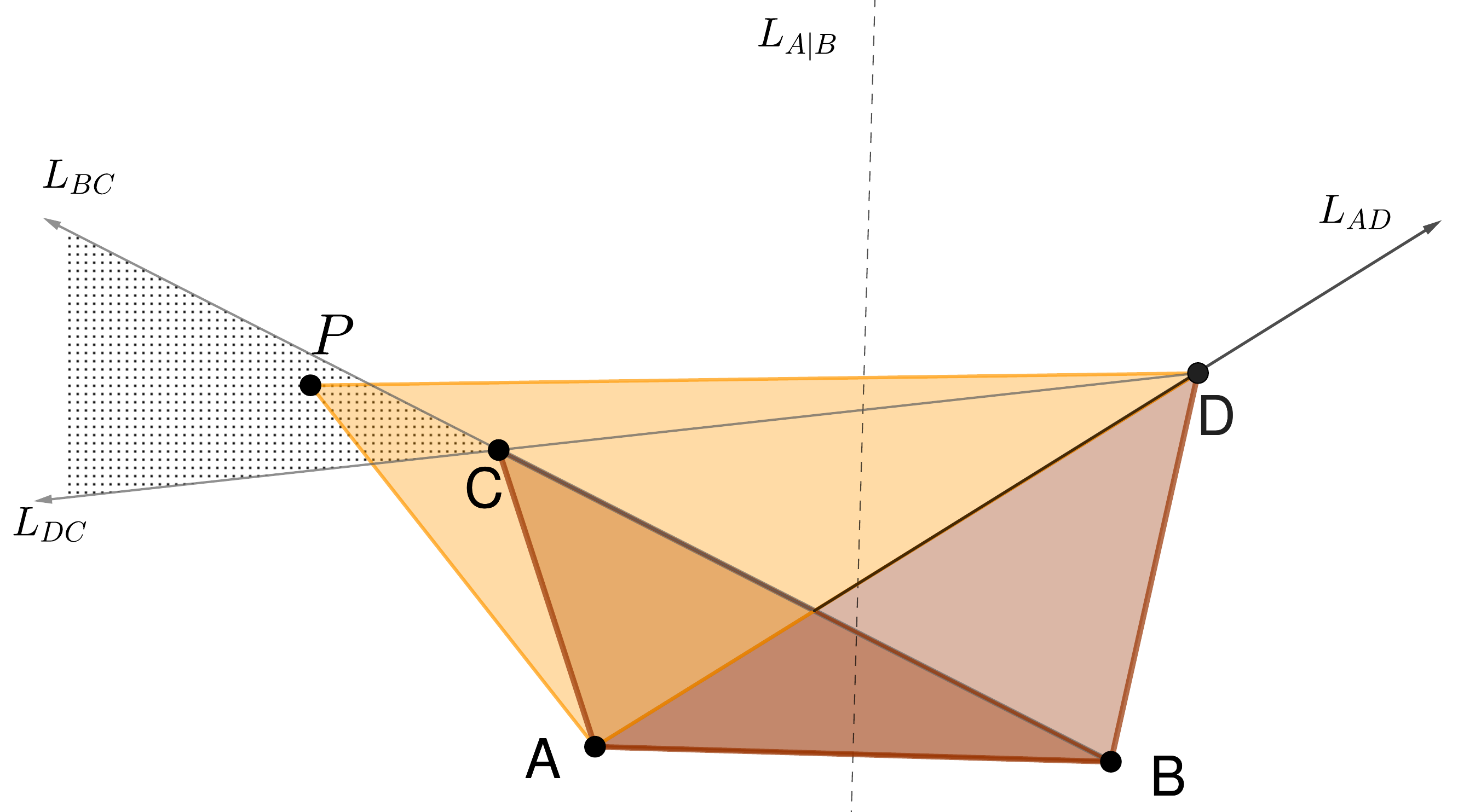}
        \caption{$P \in L^+_{A|B} \cap \ \mathrm{cl}(L^+_{BC}) \cap L^+_{DC}$} \label{m2}
  \end{subfigure}
\end{figure}
    We have $P \in L^+_{A|B}$. Now for any point $Z \in L^+_{A|B} \cap \ \mathrm{cl}(L^+_{BC})$, there are two possibilities:
    \begin{enumerate}[label=(\alph*)]
        \item  If $Z \in L^-_{DC}$ then $C \in conv \left\{A,D,Z\right\}$ (Lemma \ref{chull}).
        \item If  $Z \in L^+_{DC}$ then the segments $\overline{ZD}$ and $\overline{BC}$ intersect and from  Proposition \ref{bedge2}, $Z$ is adjacent to $B$, since $C$ and $D$ are non-adjacent.
        \end{enumerate}
        In particular, if $P\in L^+_{A|B} \cap (L_{BC} \cup L^+_{BC})$ then either $C$ is adjacent to $D$ or $P$ is adjacent to $B$, neither of which is true. This is illustrated with an example in Figure \ref{m1} and Figure \ref{m2}.  Hence, $P \in L^-_{BC}$. With analogous arguments, $ Q \in L^+_{AD}$. Since $P \in L^+_{A|B} \cap L^-_{BC} \subset L^+_{AD}$ and $Q \in L^-_{A|B} \cap L^+_{AD} \subset L^-_{BC} \ $, therefore $P, \ Q \in L^+_{AD} \cap L^-_{BC}$. Hence, Claim \ref{c2} is true.
    
\end{proof}
   
   \begin{figure}[h!]
   \centering
  \begin{subfigure}[b]{0.55\textwidth}
    \centering \includegraphics[scale=0.44]{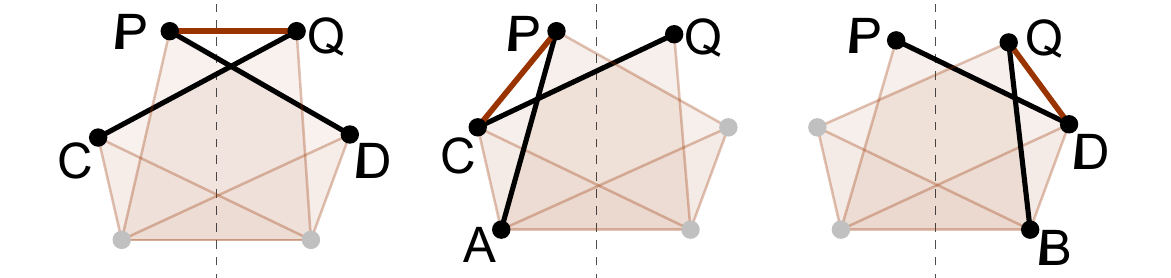}
    \caption{}
  \end{subfigure}
  \begin{subfigure}[b]{0.44\textwidth}
  \centering
    \includegraphics[scale=0.44]{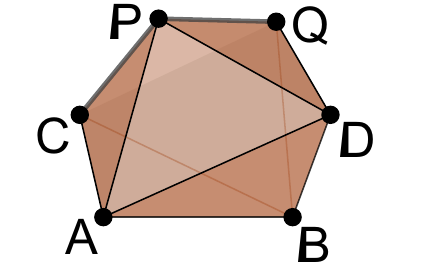}
    \caption{}
  \end{subfigure}
  \caption{ On the left, the intersection of segments (in black) in each case is possible  since $P,Q \in L^+_{AD} \cap L^-_{BC}$. On the right, the final induced complex obtained is isomorphic to $\O_2$} \label{crossings}
\end{figure}

   Now it remains to show that $P$ and $Q$ are adjacent, and both are adjacent to $C$ and $D$ which proves that the induced subcomplex $\sub{ABCDPQ}$ is isomorphic to $\O_2$. Consider the line segments $\overline{PD}$ and $\overline{QC}$; From the above observations about the region where points $P$ and $Q$ lie, the line segments $\overline{PD}$ and $\overline{QC}$ intersect since $P \notin conv \left\{ B,C,Q \right \}$ and $Q \notin conv \left\{ A,D,P \right \}$.  Likewise, the line segments $\overline{PA}$ and $\overline{QC}$ intersect, and the line segments $\overline{PD}$ and $\overline{QB}$ intersect. Along with this, the pairs $\left\{A,Q\right\}$, $\left\{B,P\right\}$, and $\left\{C,D\right\}$ are non-adjacent pairs of vertices (See Figure \ref{crossings}). Therefore it follows from  Proposition \ref{bedge2} that the points $P$ and $Q$ are adjacent, points $P$ and $C$ are adjacent, and points $Q$ and $D$ are adjacent. This concludes the proof.
   
\end{proof}

\begin{corollary} \label{mincycle}
    For a finite subset $X \subset \R^2$, if $\RR(X)$ is a minimal $2$-cycle then it is isomorphic to $\O_2$.
\end{corollary}

\begin{remark}
    It can be seen that \rm{Theorem \ref{pseudorips} for $n=2$ case is now simply a consequence of Theorem \ref{minimal} since a pseudomanifold is a minimal cycle. In general, we like to know if Theorem \ref{minimal} can be extended even for higher dimensions.}
\end{remark}

\begin{question} \label{ncycle}
    Suppose $\RR(X)$ is a minimial $n$-cycle, then is it isomorphic to $\O_n$?
\end{question}
The proof of Theorem \ref{minimal} has a further implication that in a two-dimensional, pure and closed Rips complex, every $\gamma_2$-configuration  further extends to $\O_2$, and is contained in precisely one copy of $\O_2$ in $\RR(X)$. Hence, every $\gamma_2$-configuration is a part of precisely one non-trivial homology $2$-cycle. This further leads us to characterize all two-dimensional, pure and closed Rips complexes by considering the projection map $p\colon\RR(X) \to \Sc(X)$.
\begin{lemma} \label{fullbedge}
    In a finite, two-dimensional, pure and closed planar-Rips complex $\RR(X)$, the projection of every boundary edge is completely contained in the boundary of $\Sc(X)$.
\end{lemma}
\begin{proof}
        Suppose not, then there exists two boundary edges in $\RR(X)$, say $[AB]$ and $[CD]$ whose projected interiors intersect in the shadow $\Sc(X)$ with $B$ and $C$ being non-adjacent. Now it follows from  Proposition \ref{bedge1} that the induced subcomplex $\sub{ABCD}$ is a $\gamma_2$-configuration, which further extends to $\O_2$ as shown in Theorem \ref{minimal}. But this is not possible since $[AB]$ and $[CD]$ are boundary edges in $\RR(X)$.
        
\end{proof}

\begin{proposition}

     A two-dimensional, pure and closed Rips complex $\RR(X)$ with $b_2(\RR(X))=1$ is isomorphic to $\O_2$.
\end{proposition}
\begin{proof}
    Since $\RR(X)$ contains a copy of $\O_2$ as an induced subcomplex, it suffices to show that there does not exist any other $2$-simplex outside or on the boundary edges of the induced subcomplex $\O_2$. But this follows from Proposition \ref{bedge} and Lemma \ref{fullbedge}.
\end{proof}

\begin{remark} \label{rep}
    For the planar-Rips structure of boundary complex of octahedron $\O_2$, consider its shadow $\Sc(\O_2)$ which is a convex hexagon under the projection map $p\colon\RR(X) \to \Sc(X)$. Then $p$ is injective on the boundary of $\O_2$, and the inverse image of every point in the interior of hexagon $\Sc(\O_2)$ has exactly two points in $\O_2$. More generally, it simply follows from Lemma \ref{degenerate} and Proposition \ref{bedge2} that for a $2$-dimensional pure and closed Rips complex $\RR(X)$, the inverse image of every point in the shadow $\Sc(X)$ has at most two points in $\RR(X)$. Alternatively from an algorithmic perspective, each copy of $\O_2$ can be identified as the tripartite graph $K_{2,2,2}$ (See Section \ref{multipartite}) in the graph of $1$-skeleton $\RR^{(1)}(X)$.
\end{remark}

\begin{proposition} \label{b0}
    In a finite, two-dimensional, pure and closed planar-Rips complex $\RR(X)$, denote $S$ to be set of all points in the shadow $\Sc(X)$ whose inverse image under the projection map $p\colon\RR(X) \to \Sc(X)$ has exactly two points. Then
    
    \begin{enumerate}[label=(\roman*)]
    \item $b_2(\RR(X))=b_0(S)$ i.e., the second Betti number $b_2(\RR(X))$=\# connected components in $S$. More specifically, $b_0$=Number of cycles of length $6$ in $\RR(X)$ which has at least two of its diagonals intersecting at their interiors.
    \item $\RR(X)$ has exactly $b_2(\RR(X))$ copies of $\O_2$ as induced subcomplex.
    \end{enumerate}
    
\end{proposition}
\begin{proof}
    Suppose $T$ is any connected component of $S$, then there exists two edges in $\RR(X)$ whose projected interiors intersect in $T$. Hence from Proposition \ref{bedge2}, the two edges induce a $\gamma_2$-configuration, and from Theorem \ref{minimal}, this configuration further extends to an induced subcomplex of $\O_2$. Since the projected interior of the hexagon $\Sc(\O_2)$ is connected, $\Sc(\O_2) \subseteq T$. However, the map $p$ is injective on the boundary edges of $\O_2$ which implies $T=\Sc(\O_2)$ i.e., every connected component of $S$ corresponds to an induced subcomplex of $\O_2$ in $\RR(X)$. Hence, $b_2(\RR(X))\geq b_0(S)$. Further, since any minimal $2$-cycle in $\RR(X)$ is isomorphic to $\O_2$ (Corollary \ref{mincycle}), it follows that $b_2(\RR(X))= b_0(S)$. This proves (i). Now (ii) is a direct consequence of (i).
\end{proof}

\section{Correlations with unit disk graphs} \label{applications}

In this section, we introduce the concept of obstructions of planar-Rips complexes, an area that remains unexplored in literature. However, it's worth noting that the idea of ‘obstructions’ is not entirely novel \cite{obs1,obs2}. Armed with a wide spectrum of forbidden planar-Rips structures, we present a few examples that illustrate a potential framework for extracting obstructions of planar-Rips complex from these forbidden structures.

We begin with unifying the study of planar-Rips complexes and unit disk graphs by establishing a one-to-one correspondence between both the categories. The functors that facilitate this correspondence are ``taking clique complex'' and ``taking $1$-skeleton'', which transform a unit disk graph to a planar-Rips complex and vice versa, respectively. For simplicity, we will assume all the graphs to be connected.

\begin{definition}[Disk graph] A \textbf{disk graph} is an intersection graph of equal radius disks in the Euclidean plane.  Equivalently, a disk graph is a geometric graph where each vertex corresponds to a point in the Euclidean plane and an edge is drawn between two points if and only if the distance between them is less than a fixed threshold, say $\delta>0$.
\end{definition}

 Note that in case of $\delta =1$, they are commonly referred to as \textbf{unit disk graphs}. Planar-Rips complexes are flag complexes that exhibit a key property: the clique complex of the graph of its $1$-skeleton is same as the original complex. It is easy to see that every disk graph can be identified as the $1$-skeleton of a planar-Rips complex. In fact, the following equivalence can be inferred.

\begin{proposition} \label{isomorphism}
    
    Given any finite set $X \subset \R^2$, the planar-Rips complex of $X$ is same as the clique complex of the disk graph, having the vertex set as $X$. 
\end{proposition}

For simplicity, we shall fix $\delta=1$ given the scale-invariant nature of our study. Henceforth, we shall denote the unit disk graph on the vertex set  $X \subset \R^2$ by $\RR^1(X)$, the $1$-skeleton of planar-Rips complex.

Consequently, there is a one-to-one correspondence, implying that both the categories are isomorphic.  This correspondence facilitates the integration of both studies, enabling the translation of the statements and results in the category of planar-Rips complexes into corresponding analogs for the category of unit disk graphs, or vice-versa. We say a graph $G$ admits a unit disk graph structure if there exists an injective map $f \colon V(G) \to \R^2$ such that $G$ is isomorphic to the unit disk graph $\RR^1(f(V))$ (we will denote this by $\D G$). In other words,  $uv$ is an edge if and only if the distance between $f(u)$ and $f(v)$ is less than $1$ in $\D G$. 

\begin{definition}
\begin{enumerate}[label=(\alph*)]
    \item A clique is said to be \textbf{maximal} if it is not contained in a clique of larger cardinality. We will denote the cardinality of the largest clique in a graph $G$ by $\omega(G)$.
    \item A graph is said to be \textbf{pure} if all its maximal cliques are of same size. 
    
\end{enumerate}
        
\end{definition}

Note that the graph of $1$-skeleton of a simplicial complex can be pure although the simplicial complex need not be pure. A simple example to illustrate this is the complex $\Delta_2 \sqcup \partial \Delta_2$, where $\Delta_2$ is a $2$-simplex and $\partial \Delta_2$ represents the boundary of $2$-simplex. However, in case of flag complexes, it can be easily checked that a flag complex is pure if and only if the graph of its $1$-skeleton is pure. More specifically, a unit disk graph is pure if and only if the corresponding planar-Rips complex is pure.

 \subsection{Graph of cross-polytopes.} \label{multipartite}
 
 The graph of cross-polytopes appears most often as the induced subgraph in pure unit disk graphs (See Corollary \ref{induced}). The graph of $1$-skeleton of the boundary complex of an $n$-dimensional cross-polytope is isomorphic to a complete $n$-partite graph with each independent set of size $2$. For instance, the graph of octahedron is isomorphic to the complete tripartite graph $K_{2,2,2}$ (See Figure \ref{tripartite}.) Henceforth, we shall also refer to these graph as \textbf{complete bi-independent $n$-partite graph}.  We will denote this graph by $K_{n \times 2}$ as defined by Brouwer et al. \cite[p. 478]{brouwer}, to mean the complete $n$-partite graph $ K_{\underbrace{2,2,\dots, 2}_{n \text{ times} }}$.    
       \begin{figure}[h!]
    \centering \includegraphics[scale=0.23]{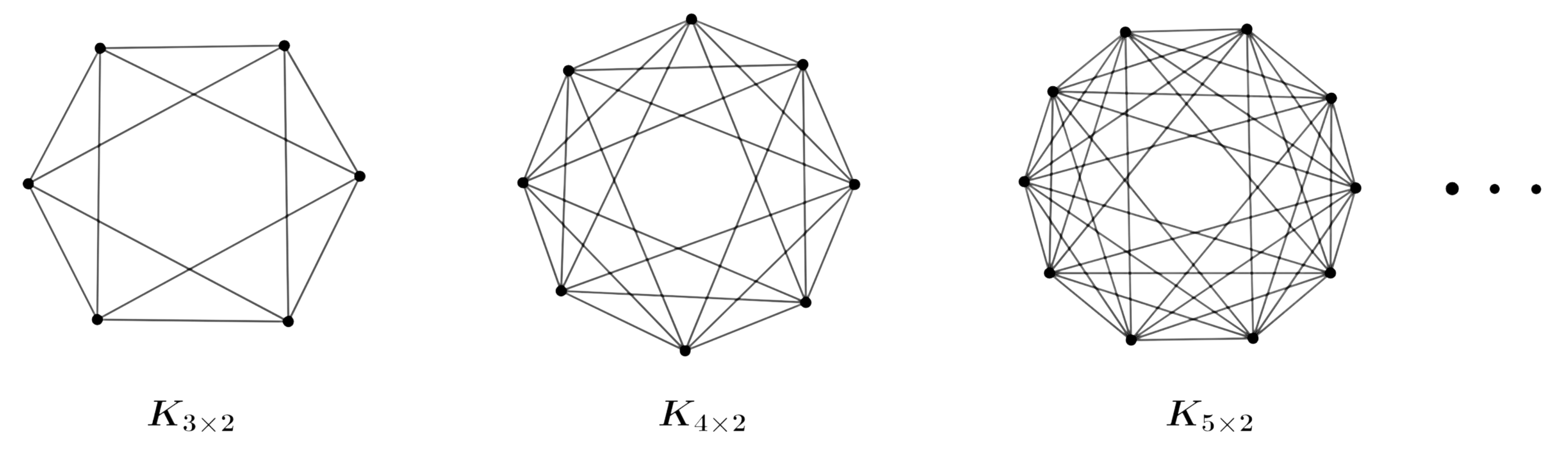}
    \caption{Complete bi-independent $n$-partite graph $K_{n \times 2}$}\label{tripartite}
\end{figure}

In the preceding sections, a complete characterization of planar-Rips complexes with $n$-dimensional weak-pseudomanifold structure and minimal $2$-cycle structure have been outlined, including a partial characterization of two-dimensional, pure, and closed planar-Rips complexes. These structures correspond to unit disk graphs having all maximal cliques of size $(n+1)$, and every clique of size $n$ contained in exactly two maximal cliques. Now we will state some of the corollaries which are natural graph analogs to these theorems proved in the previous sections.
\begin{corollary} \label{induced}
     Let $G$ be a connected graph having clique number $\omega(G)=3$ and every edge is contained in at least two maximal cliques. If $G$ admits a unit disk graph structure, then 
     \begin{enumerate}[label=(\roman*)]
    \item  $G$ contains an induced subgraph isomorphic to $K_{2,2,2}$.
    \item Given two edges $e_1$ and $e_2$ intersecting at their interiors, there exists precisely one copy of $K_{2,2,2}$ as induced subgraph that contains $e_1$ and $e_2$.
    \item Any induced subgraph $K_{2,2,2}$ in $G$ has a non-planar configuration(as shown in Figure \ref{tripartite}), containing six pairs of intersecting edges.
    \item There are precisely $\frac{k}{6}$ distinct copies of $K_{2,2,2}$ as induced subgraphs where $k$ is the number of pairs of intersecting  edges
    \end{enumerate}

\end{corollary}
\begin{proof}
    Suppose $\RR^1(X)$ is the unit disk graph representation of $G$ for some $X \subset \R^2$. Since $\omega(G)=3$, the planar-Rips complex $\RR(X)$ is two-dimensional. Further, $\RR(X)$ is pure and closed as follows from the hypothesis that $G$ is connected, and every edge is contained in at least two maximal cliques. Now the conclusion simply follows from Theorem \ref{minimal}, Proposition \ref{bedge},  Remark \ref{rep}, and Proposition \ref{b0}, where the graph of $1$-skeleton of octahedron is $K_{2,2,2}$.
\end{proof}

\begin{corollary}
    For $n \geq 3$, suppose $G$ is a pure unit disk graph with $\omega(G)=(n+1)$, and every clique of size $n$ is contained in exactly two maximal cliques. Then $G$ is isomorphic to the $1$-skeleton of an iterated chain of $\O_n$s. More precisely, up to isomorphism, $G$ can be written as $\cup_{i=1}^m H_i$, where each $H_i$ is isomorphic to $K_{n \times 2}$, no three of them have a non-empty pairwise intersection, and any two of them with non-empty intersection, intersect
     \begin{enumerate}
         \item at a single vertex when $n=2$.
         \item either at a single vertex or a single edge when $n>2$.
     \end{enumerate}
\end{corollary}
\begin{proof}
    This simply follows from Theorem \ref{weakiswedge} by considering the graph of $1$-skeleton of planar-Rips complexes, having weak-pseudomanifold structure.
\end{proof}

\subsection{Obstructions of Planar-Rips complexes}

 Given a hereditary property $\mathcal{P}$ in the category of simplicial complexes or the category of graphs, any complex or a graph that does not satisfy $\mathcal{P}$ further contains a minimal induced substructure, say $F$, that does not satisfy $\mathcal{P}$. The minimality is defined such that all the induced substructures of $F$ do satisfy $\mathcal{P}$. The existence of such an $F$ can be easily verified with repetitive vertex deletions. In case of simplicial complexes, these structures are referred to as obstructions to the corresponding property. We will now define obstructions of planar-Rips complexes, where the property $\mathcal{P}$ denotes whether a simplicial complex admits a planar-Rips structure.

\begin{definition}
    An \textit{obstruction of planar-Rips complex} is a simplicial complex that does not admit planar-Rips structure, all of whose proper induced subcomplexes do.
\end{definition}

 In case of unit disk graphs, the analog of obstructions is  minimal forbidden induced subgraphs for unit disk graphs, referred to as \textit{minimal non unit disk graphs} \cite{udg1}.

\begin{definition}
    A graph $G$ is said to be a minimal non unit disk graph (minimal non-UDG) if every induced subgraph of $G$ other than $G$ itself admits a unit disk graph structure. 
\end{definition}

There is a close relationship between the minimal non-UDGs and the obstructions of planar-Rips complexes. The following observation is a straightforward inference:
\begin{observation*}
    The clique complex of a minimal forbidden induced subgraph for unit disk graph is an obstruction of planar-Rips complex.
\end{observation*}

However, it is not true that the $1$-skeleton of every obstruction of planar-Rips complex is a minimal forbidden induced subgraph for unit disk graph. For instance, the boundary complex of all $n$-dimensional simplices are obstructions of planar-Rips complex. Considering that these are also the obstructions of flag complexes, therefore, we shall restrict only to the class of flag complexes to have the one-to-one correspondence as follows:

\begin{proposition}
    Let $K$ be a flag complex, then $K$ is an obstruction of planar-Rips complex if and only if the graph of $1$-skeleton $K^{(1)}$ is a minimal non-UDG.
\end{proposition}
\begin{proof}
    This is simply a direct consequence of Proposition \ref{isomorphism}.
\end{proof}

Therefore, it follows that the clique complex of a minimal non-UDG is an obstruction of planar-Rips complex. The literature on minimal non-UDGs is surprisingly sparse despite active research in unit disk graphs. The following theorem from \cite{udg1,udg2} provides a concise summary of currently known obstructions of planar-Rips complexes in terms of minimal non-UDGs.

 \begin{figure}[h!]
	\centering

\tikzset{every picture/.style={line width=0.75pt}}

\begin{tikzpicture}[x=0.75pt,y=0.75pt,yscale=-1,xscale=1]
 
\draw    (216.52,110.8) -- (199.19,73.3) ;
\draw [shift={(216.52,110.8)}, rotate = 245.19] [color={rgb, 255:red, 0; green, 0; blue, 0 }  ][fill={rgb, 255:red, 0; green, 0; blue, 0 }  ][line width=0.75]      (0, 0) circle [x radius= 3.35, y radius= 3.35]   ;

\draw    (251.78,110.18) -- (269.7,72.98) ;
\draw [shift={(269.7,72.98)}, rotate = 295.73] [color={rgb, 255:red, 0; green, 0; blue, 0 }  ][fill={rgb, 255:red, 0; green, 0; blue, 0 }  ][line width=0.75]      (0, 0) circle [x radius= 3.35, y radius= 3.35]   ;
\draw [shift={(251.78,110.18)}, rotate = 295.73] [color={rgb, 255:red, 0; green, 0; blue, 0 }  ][fill={rgb, 255:red, 0; green, 0; blue, 0 }  ][line width=0.75]      (0, 0) circle [x radius= 3.35, y radius= 3.35]   ;

\draw    (199.19,73.3) -- (234.44,48.92) ;
\draw [shift={(234.44,48.92)}, rotate = 325.34] [color={rgb, 255:red, 0; green, 0; blue, 0 }  ][fill={rgb, 255:red, 0; green, 0; blue, 0 }  ][line width=0.75]      (0, 0) circle [x radius= 3.35, y radius= 3.35]   ;
\draw [shift={(199.19,73.3)}, rotate = 325.34] [color={rgb, 255:red, 0; green, 0; blue, 0 }  ][fill={rgb, 255:red, 0; green, 0; blue, 0 }  ][line width=0.75]      (0, 0) circle [x radius= 3.35, y radius= 3.35]   ;

\draw    (269.7,72.98) -- (234.44,48.92) ;
\draw    (199.19,73.3) -- (234.74,84.86) ;
\draw    (234.74,84.86) -- (267.64,72.83) ;
\draw [shift={(234.74,84.86)}, rotate = 339.91] [color={rgb, 255:red, 0; green, 0; blue, 0 }  ][fill={rgb, 255:red, 0; green, 0; blue, 0 }  ][line width=0.75]      (0, 0) circle [x radius= 3.35, y radius= 3.35]   ;

\draw    (473.41,109.65) -- (473.41,72.46) ;
\draw [shift={(473.41,109.65)}, rotate = 270] [color={rgb, 255:red, 0; green, 0; blue, 0 }  ][fill={rgb, 255:red, 0; green, 0; blue, 0 }  ][line width=0.75]      (0, 0) circle [x radius= 3.35, y radius= 3.35]   ;

\draw    (508.81,109.65) -- (508.81,72.46) ;
\draw [shift={(508.81,109.65)}, rotate = 270] [color={rgb, 255:red, 0; green, 0; blue, 0 }  ][fill={rgb, 255:red, 0; green, 0; blue, 0 }  ][line width=0.75]      (0, 0) circle [x radius= 3.35, y radius= 3.35]   ;

\draw    (543.33,109.34) -- (543.33,72.15) ;
\draw [shift={(543.33,109.34)}, rotate = 270] [color={rgb, 255:red, 0; green, 0; blue, 0 }  ][fill={rgb, 255:red, 0; green, 0; blue, 0 }  ][line width=0.75]      (0, 0) circle [x radius= 3.35, y radius= 3.35]   ;

\draw    (473.41,72.46) -- (508.81,72.46) ;
\draw [shift={(508.81,72.46)}, rotate = 0] [color={rgb, 255:red, 0; green, 0; blue, 0 }  ][fill={rgb, 255:red, 0; green, 0; blue, 0 }  ][line width=0.75]      (0, 0) circle [x radius= 3.35, y radius= 3.35]   ;
\draw [shift={(473.41,72.46)}, rotate = 0] [color={rgb, 255:red, 0; green, 0; blue, 0 }  ][fill={rgb, 255:red, 0; green, 0; blue, 0 }  ][line width=0.75]      (0, 0) circle [x radius= 3.35, y radius= 3.35]   ;

\draw    (543.33,72.15) -- (508.81,72.46) ;
\draw [shift={(543.33,72.15)}, rotate = 179.48] [color={rgb, 255:red, 0; green, 0; blue, 0 }  ][fill={rgb, 255:red, 0; green, 0; blue, 0 }  ][line width=0.75]      (0, 0) circle [x radius= 3.35, y radius= 3.35]   ;

\draw    (216.52,110.8) -- (251.78,110.18) ;

\draw    (473.41,109.65) -- (508.81,109.65) ;
 
\draw    (508.81,109.65) -- (543.33,109.34) ;

\draw    (473.41,72.46) -- (508.67,48.08) ;
\draw [shift={(508.67,48.08)}, rotate = 325.34] [color={rgb, 255:red, 0; green, 0; blue, 0 }  ][fill={rgb, 255:red, 0; green, 0; blue, 0 }  ][line width=0.75]      (0, 0) circle [x radius= 3.35, y radius= 3.35]   ;

\draw    (543.33,72.15) -- (508.67,48.08) ;
 
\draw    (189.85,233.02) -- (208.08,203.11) ;
\draw [shift={(208.08,203.11)}, rotate = 301.37] [color={rgb, 255:red, 0; green, 0; blue, 0 }  ][fill={rgb, 255:red, 0; green, 0; blue, 0 }  ][line width=0.75]      (0, 0) circle [x radius= 3.35, y radius= 3.35]   ;
\draw [shift={(189.85,233.02)}, rotate = 301.37] [color={rgb, 255:red, 0; green, 0; blue, 0 }  ][fill={rgb, 255:red, 0; green, 0; blue, 0 }  ][line width=0.75]      (0, 0) circle [x radius= 3.35, y radius= 3.35]   ;

\draw    (149.51,170.99) -- (190.24,170.62) ;
\draw [shift={(149.51,170.99)}, rotate = 359.48] [color={rgb, 255:red, 0; green, 0; blue, 0 }  ][fill={rgb, 255:red, 0; green, 0; blue, 0 }  ][line width=0.75]      (0, 0) circle [x radius= 3.35, y radius= 3.35]   ;

\draw    (131.19,202.97) -- (149.51,232.65) ;
\draw [shift={(149.51,232.65)}, rotate = 58.32] [color={rgb, 255:red, 0; green, 0; blue, 0 }  ][fill={rgb, 255:red, 0; green, 0; blue, 0 }  ][line width=0.75]      (0, 0) circle [x radius= 3.35, y radius= 3.35]   ;

\draw    (149.51,170.99) -- (131.19,202.97) ;
\draw [shift={(131.19,202.97)}, rotate = 119.8] [color={rgb, 255:red, 0; green, 0; blue, 0 }  ][fill={rgb, 255:red, 0; green, 0; blue, 0 }  ][line width=0.75]      (0, 0) circle [x radius= 3.35, y radius= 3.35]   ;

\draw    (190.24,170.62) -- (208.08,203.11) ;
\draw [shift={(190.24,170.62)}, rotate = 61.22] [color={rgb, 255:red, 0; green, 0; blue, 0 }  ][fill={rgb, 255:red, 0; green, 0; blue, 0 }  ][line width=0.75]      (0, 0) circle [x radius= 3.35, y radius= 3.35]   ;

\draw    (149.51,232.65) -- (189.85,233.02) ;

\draw    (131.19,202.97) -- (172.01,202.97) ;

\draw    (172.01,202.97) -- (208.08,202.97) ;
\draw [shift={(172.01,202.97)}, rotate = 0] [color={rgb, 255:red, 0; green, 0; blue, 0 }  ][fill={rgb, 255:red, 0; green, 0; blue, 0 }  ][line width=0.75]      (0, 0) circle [x radius= 3.35, y radius= 3.35]   ;

\draw    (342.43,232.22) -- (360.66,202.32) ;
\draw [shift={(360.66,202.32)}, rotate = 301.37] [color={rgb, 255:red, 0; green, 0; blue, 0 }  ][fill={rgb, 255:red, 0; green, 0; blue, 0 }  ][line width=0.75]      (0, 0) circle [x radius= 3.35, y radius= 3.35]   ;
\draw [shift={(342.43,232.22)}, rotate = 301.37] [color={rgb, 255:red, 0; green, 0; blue, 0 }  ][fill={rgb, 255:red, 0; green, 0; blue, 0 }  ][line width=0.75]      (0, 0) circle [x radius= 3.35, y radius= 3.35]   ;

\draw    (302.08,170.2) -- (342.82,169.83) ;
\draw [shift={(302.08,170.2)}, rotate = 359.48] [color={rgb, 255:red, 0; green, 0; blue, 0 }  ][fill={rgb, 255:red, 0; green, 0; blue, 0 }  ][line width=0.75]      (0, 0) circle [x radius= 3.35, y radius= 3.35]   ;

\draw    (283.77,202.18) -- (302.08,231.86) ;
\draw [shift={(302.08,231.86)}, rotate = 58.32] [color={rgb, 255:red, 0; green, 0; blue, 0 }  ][fill={rgb, 255:red, 0; green, 0; blue, 0 }  ][line width=0.75]      (0, 0) circle [x radius= 3.35, y radius= 3.35]   ;
 
\draw    (302.08,170.2) -- (283.77,202.18) ;
\draw [shift={(283.77,202.18)}, rotate = 119.8] [color={rgb, 255:red, 0; green, 0; blue, 0 }  ][fill={rgb, 255:red, 0; green, 0; blue, 0 }  ][line width=0.75]      (0, 0) circle [x radius= 3.35, y radius= 3.35]   ;
 
\draw    (342.82,169.83) -- (360.66,202.32) ;
\draw [shift={(342.82,169.83)}, rotate = 61.22] [color={rgb, 255:red, 0; green, 0; blue, 0 }  ][fill={rgb, 255:red, 0; green, 0; blue, 0 }  ][line width=0.75]      (0, 0) circle [x radius= 3.35, y radius= 3.35]   ;

\draw    (302.08,231.86) -- (342.43,232.22) ;

\draw    (283.77,202.18) -- (333.51,202.18) ;

\draw    (333.89,202.18) -- (360.66,202.18) ;
\draw [shift={(333.89,202.18)}, rotate = 0] [color={rgb, 255:red, 0; green, 0; blue, 0 }  ][fill={rgb, 255:red, 0; green, 0; blue, 0 }  ][line width=0.75]      (0, 0) circle [x radius= 3.35, y radius= 3.35]   ;

\draw    (190.24,170.62) -- (189.85,233.02) ;

\draw    (342.82,169.83) -- (302.08,231.86) ;

\draw    (482.77,232.36) -- (501,202.45) ;
\draw [shift={(501,202.45)}, rotate = 301.37] [color={rgb, 255:red, 0; green, 0; blue, 0 }  ][fill={rgb, 255:red, 0; green, 0; blue, 0 }  ][line width=0.75]      (0, 0) circle [x radius= 3.35, y radius= 3.35]   ;
\draw [shift={(482.77,232.36)}, rotate = 301.37] [color={rgb, 255:red, 0; green, 0; blue, 0 }  ][fill={rgb, 255:red, 0; green, 0; blue, 0 }  ][line width=0.75]      (0, 0) circle [x radius= 3.35, y radius= 3.35]   ;

\draw    (442.42,170.33) -- (483.16,169.96) ;
\draw [shift={(442.42,170.33)}, rotate = 359.48] [color={rgb, 255:red, 0; green, 0; blue, 0 }  ][fill={rgb, 255:red, 0; green, 0; blue, 0 }  ][line width=0.75]      (0, 0) circle [x radius= 3.35, y radius= 3.35]   ;

\draw    (424.11,202.32) -- (442.42,231.99) ;
\draw [shift={(442.42,231.99)}, rotate = 58.32] [color={rgb, 255:red, 0; green, 0; blue, 0 }  ][fill={rgb, 255:red, 0; green, 0; blue, 0 }  ][line width=0.75]      (0, 0) circle [x radius= 3.35, y radius= 3.35]   ;

\draw    (442.42,170.33) -- (424.11,202.32) ;
\draw [shift={(424.11,202.32)}, rotate = 119.8] [color={rgb, 255:red, 0; green, 0; blue, 0 }  ][fill={rgb, 255:red, 0; green, 0; blue, 0 }  ][line width=0.75]      (0, 0) circle [x radius= 3.35, y radius= 3.35]   ;
 
\draw    (483.16,169.96) -- (501,202.45) ;
\draw [shift={(483.16,169.96)}, rotate = 61.22] [color={rgb, 255:red, 0; green, 0; blue, 0 }  ][fill={rgb, 255:red, 0; green, 0; blue, 0 }  ][line width=0.75]      (0, 0) circle [x radius= 3.35, y radius= 3.35]   ;

\draw    (442.42,231.99) -- (482.77,232.36) ;
 
\draw    (442.42,170.33) -- (462.4,202.45) ;
 
\draw    (462.4,202.45) -- (501,202.45) ;
\draw [shift={(462.4,202.45)}, rotate = 0] [color={rgb, 255:red, 0; green, 0; blue, 0 }  ][fill={rgb, 255:red, 0; green, 0; blue, 0 }  ][line width=0.75]      (0, 0) circle [x radius= 3.35, y radius= 3.35]   ;

\draw    (372.65,78.7) -- (394.37,111.52) ;
\draw [shift={(394.37,111.52)}, rotate = 56.51] [color={rgb, 255:red, 0; green, 0; blue, 0 }  ][fill={rgb, 255:red, 0; green, 0; blue, 0 }  ][line width=0.75]      (0, 0) circle [x radius= 3.35, y radius= 3.35]   ;
 
\draw    (372.65,78.7) -- (411.07,78.26) ;
\draw [shift={(411.07,78.26)}, rotate = 359.34] [color={rgb, 255:red, 0; green, 0; blue, 0 }  ][fill={rgb, 255:red, 0; green, 0; blue, 0 }  ][line width=0.75]      (0, 0) circle [x radius= 3.35, y radius= 3.35]   ;
 
\draw    (349.83,43.22) -- (372.65,78.7) ;
\draw [shift={(372.65,78.7)}, rotate = 57.25] [color={rgb, 255:red, 0; green, 0; blue, 0 }  ][fill={rgb, 255:red, 0; green, 0; blue, 0 }  ][line width=0.75]      (0, 0) circle [x radius= 3.35, y radius= 3.35]   ;
\draw [shift={(349.83,43.22)}, rotate = 57.25] [color={rgb, 255:red, 0; green, 0; blue, 0 }  ][fill={rgb, 255:red, 0; green, 0; blue, 0 }  ][line width=0.75]      (0, 0) circle [x radius= 3.35, y radius= 3.35]   ;

\draw    (372.65,78.7) -- (354.28,111.67) ;
\draw [shift={(354.28,111.67)}, rotate = 119.13] [color={rgb, 255:red, 0; green, 0; blue, 0 }  ][fill={rgb, 255:red, 0; green, 0; blue, 0 }  ][line width=0.75]      (0, 0) circle [x radius= 3.35, y radius= 3.35]   ;
 
\draw    (372.65,78.7) -- (333.96,78.55) ;
\draw [shift={(333.96,78.55)}, rotate = 180.22] [color={rgb, 255:red, 0; green, 0; blue, 0 }  ][fill={rgb, 255:red, 0; green, 0; blue, 0 }  ][line width=0.75]      (0, 0) circle [x radius= 3.35, y radius= 3.35]   ;
 
\draw    (393.53,42.92) -- (372.65,78.7) ;
\draw [shift={(393.53,42.92)}, rotate = 120.26] [color={rgb, 255:red, 0; green, 0; blue, 0 }  ][fill={rgb, 255:red, 0; green, 0; blue, 0 }  ][line width=0.75]      (0, 0) circle [x radius= 3.35, y radius= 3.35]   ;

\draw    (109.8,110.8) -- (80,80.3) ;

\draw    (109.8,110.8) -- (140.4,80.7) ;
\draw [shift={(140.4,80.7)}, rotate = 315.47] [color={rgb, 255:red, 0; green, 0; blue, 0 }  ][fill={rgb, 255:red, 0; green, 0; blue, 0 }  ][line width=0.75]      (0, 0) circle [x radius= 3.35, y radius= 3.35]   ;
\draw [shift={(109.8,110.8)}, rotate = 315.47] [color={rgb, 255:red, 0; green, 0; blue, 0 }  ][fill={rgb, 255:red, 0; green, 0; blue, 0 }  ][line width=0.75]      (0, 0) circle [x radius= 3.35, y radius= 3.35]   ;
 
\draw    (80,80.3) -- (109.6,50.3) ;
\draw [shift={(109.6,50.3)}, rotate = 314.62] [color={rgb, 255:red, 0; green, 0; blue, 0 }  ][fill={rgb, 255:red, 0; green, 0; blue, 0 }  ][line width=0.75]      (0, 0) circle [x radius= 3.35, y radius= 3.35]   ;
\draw [shift={(80,80.3)}, rotate = 314.62] [color={rgb, 255:red, 0; green, 0; blue, 0 }  ][fill={rgb, 255:red, 0; green, 0; blue, 0 }  ][line width=0.75]      (0, 0) circle [x radius= 3.35, y radius= 3.35]   ;

\draw    (140.4,80.7) -- (109.6,50.3) ;

\draw    (80,80.3) -- (110.4,80.5) ;

\draw    (110.4,80.5) -- (140.4,80.7) ;
\draw [shift={(110.4,80.5)}, rotate = 0.38] [color={rgb, 255:red, 0; green, 0; blue, 0 }  ][fill={rgb, 255:red, 0; green, 0; blue, 0 }  ][line width=0.75]      (0, 0) circle [x radius= 3.35, y radius= 3.35]   ;

\draw (222,123.4) node [anchor=north west][inner sep=0.75pt]    {$G_{2}$};

\draw (452.67,244.4) node [anchor=north west][inner sep=0.75pt]    {$G_{7}$};

\draw (308,245.73) node [anchor=north west][inner sep=0.75pt]    {$G_{6}$};

\draw (155.33,246.73) node [anchor=north west][inner sep=0.75pt]    {$G_{5}$};

\draw (498.22,121.89) node [anchor=north west][inner sep=0.75pt]    {$G_{4}$};

\draw (362,122.07) node [anchor=north west][inner sep=0.75pt]    {$G_{3}$};

\draw (99,123.4) node [anchor=north west][inner sep=0.75pt]    {$G_{1}$};
\end{tikzpicture}

	\caption{Known minimal non unit disk graphs with $\leq7$ vertices (Source: \cite[p. 2]{udg2})}
	\label{fig:knownMinForb}
\end{figure}

\begin{theorem} \label{nonudg}
    The following graphs are minimal non-UDGs. More precisely, the clique complex of these graphs are obstructions of planar-Rips complex.
    \begin{enumerate}
        \item Graphs in Figure \ref{fig:knownMinForb}
        \item $\overline{K_2+C_{2k+1}}$ for every $k\geq 1$ 
        \item   $\overline{C_{2k}}$ for every $k\geq 4$
        \item   $C^{*}_{2k}$ for every $k\geq 4$
    \end{enumerate}
\end{theorem}
Note that in the above theorem, the graph $C^{*}_{2k}$ denotes the graph obtained by adding two $k$-cliques $\overline{U}$ and $\overline{V}$ to the cyclic graph $C_{2k}$, where $U$ and $V$ are the two maximal independent sets in $C_{2k}$, each of size $k$. Furthermore, it follows that there is an obstruction of planar-Rips complex in every odd dimension $d$.

In the realm of topology, the classification problem, namely, classification of structures like manifolds, pseudomanifolds and weak-pseudomanifolds up to homeomorphism (or homotopy) has been extensively studied and still remains a hot topic. Indeed, these structures have been largely classified in the literature \cite{pseudo,bagchi,brown,wall}. Following Theorem \ref{free}, Theorem \ref{minimal} and Lemma \ref{pseudo}, we have a large class of forbidden induced structures for planar-Rips complexes and unit disk graphs.  For instance, the result by Adamaszek et al. in Theorem \ref{normal} implies that no manifold other than the cross-polytopal spheres $\O_n$ admits a planar-Rips structure. This makes a pavement for extracting a wide spectrum of obstructions from these forbidden structures. Consider the following examples.
\begin{example} \label{k2}
   Let $K$ be the complex $C_5 * \s^0$, a flag triangulation of sphere $\s^2$ where $C_5$ is a cycle of length $5$. The complex $K$ being a $2$-dimensional pseudomanifold does not admit a planar-Rips structure as follows from Theorem \ref{octa}. To determine whether $K$ is an obstruction of a planar-Rips complex, we note that the $1$-skeleton of $K$ can be identified as the complement of the graph $\overline{K_2 + C_5}$ which is a minimal non-UDG as shown in \cite[Theorem 4.1]{udg2}. Therefore, the complex $K$ is indeed an obstruction of planar-Rips complex.
\end{example}

 \begin{figure}[h!]
  \begin{subfigure}[b]{0.45\textwidth}
    \centering

\tikzset{every picture/.style={line width=0.75pt}} 

\begin{tikzpicture}[x=0.75pt,y=0.75pt,yscale=-0.75,xscale=0.75]

\draw [line width=1.5]    (341.7,108.43) -- (341.7,154.12) ;

\draw [line width=1.5]    (341.7,63.46) -- (341.7,108.43) ;

\draw [color={rgb, 255:red, 120; green, 120; blue, 120 }  ,draw opacity=1 ]   (372.15,194.57) -- (414.03,177.67) ;

\draw [color={rgb, 255:red, 120; green, 120; blue, 120 }  ,draw opacity=1 ]   (296.46,140.77) -- (272.19,176.7) ;
 
\draw [color={rgb, 255:red, 120; green, 120; blue, 120 }  ,draw opacity=1 ]   (272.19,176.7) -- (313.6,192.64) ;
 
\draw [color={rgb, 255:red, 120; green, 120; blue, 120 }  ,draw opacity=1 ]   (254.11,127.68) -- (296.46,140.77) ;
 
\draw [color={rgb, 255:red, 120; green, 120; blue, 120 }  ,draw opacity=1 ]   (295.54,97.01) -- (341.7,108.43) ;
 
\draw    (295.54,97.01) -- (296.96,141.27) ;

\draw [color={rgb, 255:red, 120; green, 120; blue, 120 }  ,draw opacity=1 ]   (295.04,96.51) -- (254.11,127.68) ;

\draw [color={rgb, 255:red, 120; green, 120; blue, 120 }  ,draw opacity=1 ][line width=0.75]    (341.2,62.96) -- (295.04,96.51) ;

\draw [color={rgb, 255:red, 120; green, 120; blue, 120 }  ,draw opacity=1 ]   (341.2,62.96) -- (385.93,94.61) ;

\draw [color={rgb, 255:red, 120; green, 120; blue, 120 }  ,draw opacity=1 ]   (386.43,95.11) -- (428.31,124.37) ;

\draw [color={rgb, 255:red, 120; green, 120; blue, 120 }  ,draw opacity=1 ]   (386.43,95.11) -- (386.43,140.08) ;
 
\draw [color={rgb, 255:red, 120; green, 120; blue, 120 }  ,draw opacity=1 ]   (428.31,124.37) -- (414.03,177.67) ;
 
\draw [color={rgb, 255:red, 120; green, 120; blue, 120 }  ,draw opacity=1 ]   (414.03,177.67) -- (399.28,228.35) ;
 
\draw [color={rgb, 255:red, 120; green, 120; blue, 120 }  ,draw opacity=1 ]   (343.58,229.28) -- (398.78,227.85) ;

\draw [color={rgb, 255:red, 120; green, 120; blue, 120 }  ,draw opacity=1 ]   (313.6,192.64) -- (290.75,229.76) ;
 
\draw [color={rgb, 255:red, 120; green, 120; blue, 120 }  ,draw opacity=1 ]   (290.75,229.76) -- (343.58,229.28) ;

\draw [color={rgb, 255:red, 120; green, 120; blue, 120 }  ,draw opacity=1 ]   (371.65,194.07) -- (398.78,227.85) ;

\draw [line width=1.5]    (372.15,194.57) -- (344.08,229.78) ;

\draw [line width=1.5]    (314.1,193.14) -- (344.55,232.64) ;
 
\draw [color={rgb, 255:red, 120; green, 120; blue, 120 }  ,draw opacity=1 ]   (254.11,127.68) -- (272.19,176.7) ;

\draw [color={rgb, 255:red, 120; green, 120; blue, 120 }  ,draw opacity=1 ]   (272.19,176.7) -- (290.75,229.76) ;
 
\draw [line width=1.5]    (341.7,154.12) -- (386.43,140.08) ;

\draw [line width=1.5]    (296.96,141.27) -- (341.7,154.12) ;

\draw [line width=1.5]    (341.7,154.12) -- (372.15,194.57) ;

\draw [line width=1.5]    (314.1,193.14) -- (372.15,194.57) ;

\draw [line width=1.5]    (341.7,108.43) -- (386.43,140.08) ;

\draw [line width=1.5]    (296.96,141.27) -- (314.1,193.14) ;

\draw [line width=1.5]    (341.7,154.12) -- (314.1,193.14) ;

\draw [line width=1.5]    (386.43,140.08) -- (372.15,194.57) ;

\draw [line width=1.5]    (341.7,108.43) -- (296.96,141.27) ;

\draw [color={rgb, 255:red, 120; green, 120; blue, 120 }  ,draw opacity=1 ]   (341.7,108.43) -- (388.81,97.01) ;

\draw [color={rgb, 255:red, 120; green, 120; blue, 120 }  ,draw opacity=1 ]   (386.43,140.08) -- (428.31,124.37) ;

\draw [color={rgb, 255:red, 120; green, 120; blue, 120 }  ,draw opacity=1 ]   (386.43,140.08) -- (414.03,177.67) ;

\draw  [color={rgb, 255:red, 247; green, 247; blue, 247 }  ,draw opacity=1 ][fill={rgb, 255:red, 255; green, 255; blue, 255 }  ,fill opacity=1 ][dash pattern={on 0.84pt off 2.51pt}]  (253.42, 126.63) circle [x radius= 12.1, y radius= 12.1]   ;
\draw (247.92,120.13) node [anchor=north west][inner sep=0.75pt]  [font=\footnotesize,color={rgb, 255:red, 0; green, 0; blue, 0 }  ,opacity=1 ] [align=left] {I};

\draw  [color={rgb, 255:red, 249; green, 249; blue, 249 }  ,draw opacity=1 ][fill={rgb, 255:red, 255; green, 255; blue, 255 }  ,fill opacity=1 ][dash pattern={on 0.84pt off 2.51pt}][line width=0.75]   (340.72, 62.23) circle [x radius= 13.2, y radius= 13.2]   ;
\draw (334.72,54.73) node [anchor=north west][inner sep=0.75pt]  [font=\small] [align=left] {C};

\draw  [color={rgb, 255:red, 249; green, 249; blue, 249 }  ,draw opacity=1 ][fill={rgb, 255:red, 255; green, 255; blue, 255 }  ,fill opacity=1 ][dash pattern={on 0.84pt off 2.51pt}]  (340.49, 152.17) circle [x radius= 12.9, y radius= 12.9]   ;
\draw (334.99,144.67) node [anchor=north west][inner sep=0.75pt]  [font=\small] [align=left] {A};

\draw  [color={rgb, 255:red, 247; green, 247; blue, 247 }  ,draw opacity=1 ][fill={rgb, 255:red, 255; green, 255; blue, 255 }  ,fill opacity=1 ][dash pattern={on 0.84pt off 2.51pt}]  (396.85, 223.48) circle [x radius= 12.1, y radius= 12.1]   ;
\draw (391.35,216.98) node [anchor=north west][inner sep=0.75pt]  [font=\footnotesize,color={rgb, 255:red, 0; green, 0; blue, 0 }  ,opacity=1 ] [align=left] {H};

\draw  [color={rgb, 255:red, 247; green, 247; blue, 247 }  ,draw opacity=1 ][fill={rgb, 255:red, 255; green, 255; blue, 255 }  ,fill opacity=1 ][dash pattern={on 0.84pt off 2.51pt}]  (290.04, 225.38) circle [x radius= 12.1, y radius= 12.1]   ;
\draw (284.54,218.88) node [anchor=north west][inner sep=0.75pt]  [font=\footnotesize,color={rgb, 255:red, 0; green, 0; blue, 0 }  ,opacity=1 ] [align=left] {K};

\draw  [color={rgb, 255:red, 249; green, 249; blue, 249 }  ,draw opacity=1 ][fill={rgb, 255:red, 255; green, 255; blue, 255 }  ,fill opacity=1 ][dash pattern={on 0.84pt off 2.51pt}]  (312.65, 190.24) circle [x radius= 13.2, y radius= 13.2]   ;
\draw (306.65,182.74) node [anchor=north west][inner sep=0.75pt]  [font=\small] [align=left] {D};

\draw  [color={rgb, 255:red, 249; green, 249; blue, 249 }  ,draw opacity=1 ][fill={rgb, 255:red, 255; green, 255; blue, 255 }  ,fill opacity=1 ][dash pattern={on 0.84pt off 2.51pt}]  (370.94, 190.71) circle [x radius= 12.9, y radius= 12.9]   ;
\draw (365.44,183.21) node [anchor=north west][inner sep=0.75pt]  [font=\small] [align=left] {E};

\draw  [color={rgb, 255:red, 247; green, 247; blue, 247 }  ,draw opacity=1 ][fill={rgb, 255:red, 255; green, 255; blue, 255 }  ,fill opacity=1 ][dash pattern={on 0.84pt off 2.51pt}]  (428.07, 121.19) circle [x radius= 12.1, y radius= 12.1]   ;
\draw (422.57,114.69) node [anchor=north west][inner sep=0.75pt]  [font=\footnotesize,color={rgb, 255:red, 0; green, 0; blue, 0 }  ,opacity=1 ] [align=left] {J};

\draw  [color={rgb, 255:red, 249; green, 249; blue, 249 }  ,draw opacity=1 ][fill={rgb, 255:red, 255; green, 255; blue, 255 }  ,fill opacity=1 ][dash pattern={on 0.84pt off 2.51pt}]  (296.7, 139.79) circle [x radius= 13.51, y radius= 13.51]   ;
\draw (290.2,132.29) node [anchor=north west][inner sep=0.75pt]  [font=\small] [align=left] {G};

\draw  [color={rgb, 255:red, 249; green, 249; blue, 249 }  ,draw opacity=1 ][fill={rgb, 255:red, 255; green, 255; blue, 255 }  ,fill opacity=1 ][dash pattern={on 0.84pt off 2.51pt}]  (385.48, 136.94) circle [x radius= 12.9, y radius= 12.9]   ;
\draw (379.98,129.44) node [anchor=north west][inner sep=0.75pt]  [font=\small] [align=left] {F};

\draw  [color={rgb, 255:red, 249; green, 249; blue, 249 }  ,draw opacity=1 ][fill={rgb, 255:red, 255; green, 255; blue, 255 }  ,fill opacity=1 ][dash pattern={on 0.84pt off 2.51pt}]  (340.49, 106.01) circle [x radius= 12.9, y radius= 12.9]   ;
\draw (334.99,98.51) node [anchor=north west][inner sep=0.75pt]  [font=\small] [align=left] {B};

\draw  [color={rgb, 255:red, 249; green, 249; blue, 249 }  ,draw opacity=1 ][fill={rgb, 255:red, 255; green, 255; blue, 255 }  ,fill opacity=1 ][dash pattern={on 0.84pt off 2.51pt}]  (343.1, 226.88) circle [x radius= 13.2, y radius= 13.2]   ;
\draw (337.1,219.38) node [anchor=north west][inner sep=0.75pt]  [font=\small] [align=left] {C};

\draw  [color={rgb, 255:red, 247; green, 247; blue, 247 }  ,draw opacity=1 ][fill={rgb, 255:red, 255; green, 255; blue, 255 }  ,fill opacity=1 ][dash pattern={on 0.84pt off 2.51pt}]  (413.32, 174.49) circle [x radius= 12.1, y radius= 12.1]   ;
\draw (407.82,167.99) node [anchor=north west][inner sep=0.75pt]  [font=\footnotesize,color={rgb, 255:red, 0; green, 0; blue, 0 }  ,opacity=1 ] [align=left] {I};

\draw  [color={rgb, 255:red, 255; green, 255; blue, 255 }  ,draw opacity=1 ][fill={rgb, 255:red, 255; green, 255; blue, 255 }  ,fill opacity=1 ][dash pattern={on 0.84pt off 2.51pt}]  (386.65, 92.63) circle [x radius= 12.1, y radius= 12.1]   ;
\draw (381.15,86.13) node [anchor=north west][inner sep=0.75pt]  [font=\footnotesize,color={rgb, 255:red, 0; green, 0; blue, 0 }  ,opacity=1 ] [align=left] {K};

\draw  [color={rgb, 255:red, 247; green, 247; blue, 247 }  ,draw opacity=1 ][fill={rgb, 255:red, 255; green, 255; blue, 255 }  ,fill opacity=1 ][dash pattern={on 0.84pt off 2.51pt}]  (295.01, 91.66) circle [x radius= 12.1, y radius= 12.1]   ;
\draw (289.51,85.16) node [anchor=north west][inner sep=0.75pt]  [font=\footnotesize,color={rgb, 255:red, 0; green, 0; blue, 0 }  ,opacity=1 ] [align=left] {H};

\draw  [color={rgb, 255:red, 247; green, 247; blue, 247 }  ,draw opacity=1 ][fill={rgb, 255:red, 255; green, 255; blue, 255 }  ,fill opacity=1 ][dash pattern={on 0.84pt off 2.51pt}]  (271.03, 173.03) circle [x radius= 12.1, y radius= 12.1]   ;
\draw (265.53,166.53) node [anchor=north west][inner sep=0.75pt]  [font=\footnotesize,color={rgb, 255:red, 0; green, 0; blue, 0 }  ,opacity=1 ] [align=left] {J};

\end{tikzpicture}

    \caption{Minimal flag triangulation of $\R P^2$ } \label{rp2}
  \end{subfigure}
  \begin{subfigure}[b]{0.45\textwidth}
    \centering

\tikzset{every picture/.style={line width=0.75pt}} 

\begin{tikzpicture}[x=0.75pt,y=0.75pt,yscale=-0.75,xscale=0.75]

\draw [line width=0.75]    (290.99,34.8) -- (293.74,169.3) ;
\draw [shift={(293.74,169.3)}, rotate = 88.83] [color={rgb, 255:red, 0; green, 0; blue, 0 }  ][fill={rgb, 255:red, 0; green, 0; blue, 0 }  ][line width=0.75]      (0, 0) circle [x radius= 3.35, y radius= 3.35]   ;

\draw [line width=0.75]    (249.4,147.66) -- (293.74,169.3) ;

\draw [line width=0.75]    (334.5,147.01) -- (293.74,169.3) ;

\draw [line width=0.75]    (355.35,74.79) -- (290.99,34.8) ;
\draw [shift={(290.99,34.8)}, rotate = 211.86] [color={rgb, 255:red, 0; green, 0; blue, 0 }  ][fill={rgb, 255:red, 0; green, 0; blue, 0 }  ][line width=0.75]      (0, 0) circle [x radius= 3.35, y radius= 3.35]   ;
\draw [shift={(355.35,74.79)}, rotate = 211.86] [color={rgb, 255:red, 0; green, 0; blue, 0 }  ][fill={rgb, 255:red, 0; green, 0; blue, 0 }  ][line width=0.75]      (0, 0) circle [x radius= 3.35, y radius= 3.35]   ;

\draw [line width=0.75]    (270.25,95.77) -- (249.4,147.66) ;

\draw [line width=0.75]    (355.35,74.79) -- (334.5,147.01) ;

\draw [line width=0.75]    (225.08,74.23) -- (249.4,147.66) ;
 
\draw [line width=0.75]    (290.99,34.8) -- (225.08,74.23) ;

\draw [line width=0.75]    (270.25,95.77) -- (334.5,147.01) ;

\draw [line width=0.75]    (270.25,95.77) -- (355.35,74.79) ;

\draw [line width=0.75]    (225.08,74.23) -- (270.25,95.77) ;
\draw [shift={(270.25,95.77)}, rotate = 25.5] [color={rgb, 255:red, 0; green, 0; blue, 0 }  ][fill={rgb, 255:red, 0; green, 0; blue, 0 }  ][line width=0.75]      (0, 0) circle [x radius= 3.35, y radius= 3.35]   ;
\draw [shift={(225.08,74.23)}, rotate = 25.5] [color={rgb, 255:red, 0; green, 0; blue, 0 }  ][fill={rgb, 255:red, 0; green, 0; blue, 0 }  ][line width=0.75]      (0, 0) circle [x radius= 3.35, y radius= 3.35]   ;

\draw [line width=0.75]    (290.99,34.8) -- (270.25,95.77) ;

\draw [line width=0.75]    (249.4,147.66) -- (334.5,147.01) ;
\draw [shift={(334.5,147.01)}, rotate = 359.56] [color={rgb, 255:red, 0; green, 0; blue, 0 }  ][fill={rgb, 255:red, 0; green, 0; blue, 0 }  ][line width=0.75]      (0, 0) circle [x radius= 3.35, y radius= 3.35]   ;
\draw [shift={(249.4,147.66)}, rotate = 359.56] [color={rgb, 255:red, 0; green, 0; blue, 0 }  ][fill={rgb, 255:red, 0; green, 0; blue, 0 }  ][line width=0.75]      (0, 0) circle [x radius= 3.35, y radius= 3.35]   ;

\draw (268.44,107.46) node [anchor=north west][inner sep=0.75pt]   [align=left] {A};

\draw (203.98,81.59) node [anchor=north west][inner sep=0.75pt]   [align=left] {G};

\draw (300.42,18.33) node [anchor=north west][inner sep=0.75pt]   [align=left] {B};

\draw (361.69,82.38) node [anchor=north west][inner sep=0.75pt]   [align=left] {F};

\draw (291.8,180.24) node [anchor=north west][inner sep=0.75pt]   [align=left] {C};

\draw (234.6,156.91) node [anchor=north west][inner sep=0.75pt]   [align=left] {D};

\draw (342.65,151.23) node [anchor=north west][inner sep=0.75pt]   [align=left] {E};

\end{tikzpicture}

        \caption{Graph of $1$-skeleton of $\sub{ABCDEFG}$} \label{obs}
  \end{subfigure}
  \caption{Representation of minimal flag triangulation of $\R P^2$ in terms of its $1$-skeleton (on the left) that does not admit a planar-Rips structure. Its induced subcomplex $\sub{ABCDEFG}$ (on the right) is an obstruction of planar-Rips complex}
\end{figure}

\begin{example}
    Consider the minimal flag triangulation of the projective space $\R P^2$ with $11$ vertices as shown in Figure \ref{rp2} \cite[Example 3.1]{rp2}. We know that this complex cannot be realized as a planar-Rips complex since its fundamental group is not free. Hence, there exists at least one obstruction of planar-Rips complex that is embedded inside this triangulation. The complex being non-orientable and with a torsion in its homology group, it appears to be a `highly forbidden' structure for a planar-Rips complex, suggesting that any embedded obstruction within it is likely to exhibit repetitive structural patterns. With this in mind, we consider the induced subcomplex $K=\sub{ABCDEFG}$, and consequently use topological methods to prove in the next theorem that $K$ is indeed an obstruction of planar-Rips complex. In other words, the graph of $1$-skeleton of $K$ is a minimal non-UDG. 
\end{example}

\begin{theorem} \label{obsinrp2}
    The graph in Figure \ref{obs} is a minimal non-UDG, or equivalently, its clique complex is an obstruction of planar-Rips complex.
\end{theorem}
\begin{proof}
    Let $K=\sub{ABCDEFG}$ denote the clique complex of the graph in Figure \ref{obs}. We first show that $K$ does not admit a planar-Rips structure. On contrary,  suppose it admits a planar-Rips structure; Without loss of generality, let $\RR(X)$ denote the planar-Rips structure, where $X=\left\{A,B,C,D,E,F,G\right\} \subset \R^2$. Note that the fundamental group $\pi_1(\RR(X))$ is non-trivial since the cycle $\sub{BCEF}$ at the vertex $B$ is not contractible in $\RR(X)$.

     \begin{figure}[h!]
  \begin{subfigure}[b]{0.45\textwidth}
    \centering 
    \includegraphics[scale=0.27]{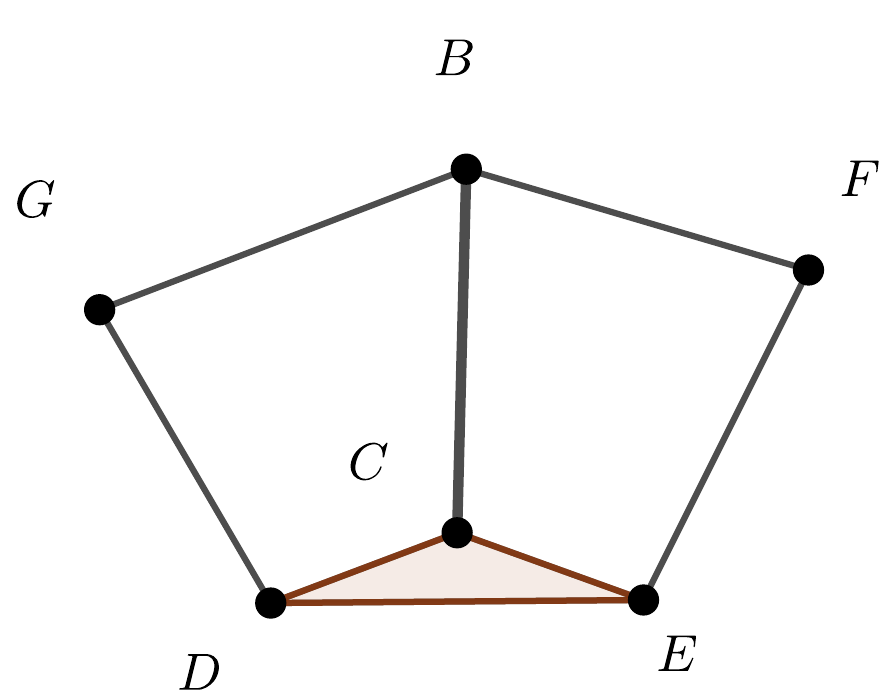}
    \caption{}
  \end{subfigure}
  \begin{subfigure}[b]{0.45\textwidth}
    \centering 
    \includegraphics[scale=0.3]{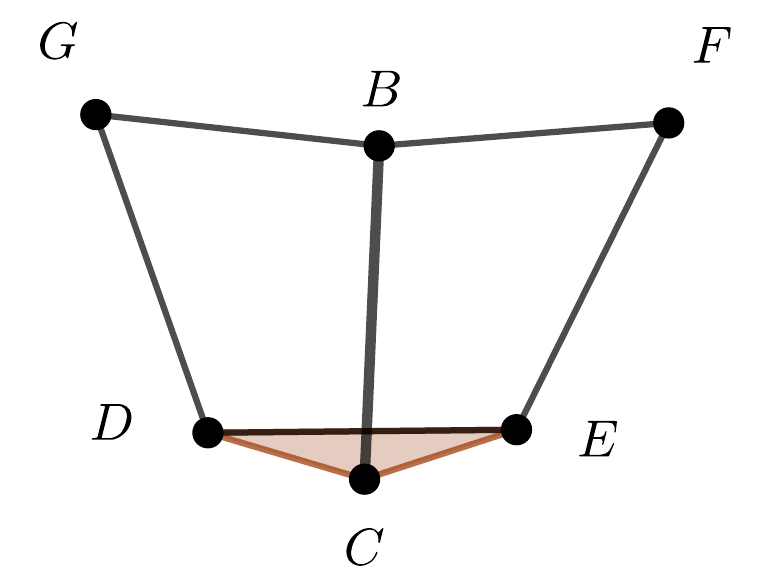}
        \caption{}
  \end{subfigure}
  \caption{ Two possible configurations of $\sub{BCDEFG}$, illustrating that $\overline{CB}$ lies in the interior of polygon $DCEFBG$.} \label{ininterior}
\end{figure}
    \begin{claim*}
        The line segment $\overline{CB} \in \conv\left\{C,D,E\right\} \cup \conv\left\{D,E,F,B,G\right\}$
    \end{claim*}
    \begin{proof} \renewcommand{\qedsymbol}{}
    Consider the line $L_{CB}$ and the perpendicular bisectors $L_{C|B}, L_{D|C}$, and $L_{C|E}$. Firstly, $E$ cannot be on the line $L_{CB}$ since $E$ is not adjacent to $B$, and $C$ is not adjacent to $F$. Likewise, $D$ cannot be on the line $L_{CB}$. Without loss of generality, assume $E$ is on the right side of the line $L_{CB}$ i.e., $E \in L^-_{CB}$. Considering the adjacency between the vertices, we have the following observations:
    \begin{enumerate}[label=(\alph*)]
        \item $D,E \in L^+_{C|B}$ and $F,G \in L^-_{C|B}$
        \item $B\in L^+_{C|E}$ and $F \in L^-_{C|E}$
        \item $B \in L^-_{D|C}$ and $G \in L^+_{D|C}$
        \item $E$ and $F$ are on same side of line $L_{CB}$
        \item $D$ and $G$ are on same side of line $L_{CB}$
    \end{enumerate}
    If suppose $D,G \in L^-_{CB}$. Since non-adjacent vertices $F$ and $G$ are adjacent to $B$, we have $G \in L^+_{B|F}$ and $F \in L^+_{B|G}$. But in contrast, this is not possible unless $F$ and $G$ are adjacent. Hence, $D,G \in L^+_{CB}$.

    With this arrangement of vertices, the segment $\overline{CB}$ lies on the left side of every line joining any two adjacent vertices of polygon $DCEFBG$ (in counterclockwise order). See Figure \ref{ininterior} for an illustration. It follows that $\overline{CB} \in \conv\left\{C,D,E\right\} \cup \conv\left\{D,E,F,B,G\right\}$. This proves the claim.
    \end{proof}
    
    Consequently, the shadow $\Sc(X)$ is contractible since $\sub{ABDEFG}$ is a cone over vertex $A$. In particular, $\pi_1(\RR(X))$ is trivial as follows from Theorem \ref{free}. This is a contradiction. Therefore, the complex $K$ does not admit a planar-Rips structure. 
    
    Additionally, it can be easily verified that every induced subcomplex of $K$ admits a planar-Rips structure. Therefore, $K$ is an obstruction of planar-Rips complex.
\end{proof}

Note that the minimal non-UDG in Figure \ref{obs} is a spanning subgraph of $\overline{K_2 + C_5}$, which is also a minimal non-UDG (see Example \ref{k2}). In particular, the class of Unit Disk graphs need not be hereditary under edge deletion.

\section{Concluding remarks}

The methods and results presented in this paper highlight the effectiveness of topological and geometric approaches towards characterizing planar-Rips complexes.  We have successfully  obtained a complete classification of planar-Rips complexes that have pseudomanifold and weak-pseudomanifold structure, demonstrating that they are isomorphic to iterated wedge of crosspolytopal spheres. Furthermore, we characterized the two-dimensional pure and closed simplicial complexes, identifying the octahedron as the only minimal 2-cycle. Perhaps, these structures have  homotopy type of wedge sum of spheres.

From an application standpoint, we have explored the isomorphism between the category of planar-Rips complexes and unit disk graphs, providing a characterization of a large class of disk graphs with uniform clique structure. In particular, the concept of obstructions in planar-Rips complexes has been introduced, a tool that offers a new framework for recognizing and classifying these structures. Notably, this concept has a natural generalization to Vietoris-Rips complexes inside any metric space. Following are some of the problems that remain open:
\begin{enumerate}[label=(\roman*)]
    \item While our findings support Adamaszek’s conjecture \cite{planar2} that the planar-Rips complexes are homotopy equivalent to wedge of spheres, the problem still remains unsolved.
    \item The minimal 2-cycle isomorphism to the boundary of the 2-crosspolytope invites further exploration in higher dimensions. Perhaps, this can be the final piece to proving the above conjecture.
\end{enumerate}
It is also worth noting that the challenge of determining the obstructions of planar-Rips complexes can be reformulated into the following problem:
\begin{prob}

    Given a flag complex $K$ that does not admit a planar-Rips structure, devise methodologies to determine the induced subcomplex(es) of $K$ that are obstructions of planar-Rips complex.
\end{prob}

Developing algorithms for recognizing forbidden planar-Rips complexes presents another significant avenue for research. It would be interesting to know how far this approach can be developed.

\end{document}